\documentclass[12pt]{amsart}

\usepackage{color}
\usepackage{enumitem}

\usepackage{amsmath,amssymb,setspace,nicefrac, yhmath, amscd,eucal}
\usepackage[active]{srcltx}
\usepackage[pagebackref, colorlinks, linkcolor=red, citecolor=blue, urlcolor=blue, hypertexnames=true]{hyperref}
\usepackage{amsrefs}
\usepackage{tikz}

\allowdisplaybreaks
\usepackage[all]{xy}

\newcommand{\R}{{\mathbb R}}

\newcommand{\Z}{{\mathbb Z}}

\newcommand{\cH}{{\mathcal H}}

\newtheorem{thm}{Theorem}[section]
\newtheorem{cor}[thm]{Corollary}
\newtheorem{lem}[thm]{Lemma}

\newtheorem{definition}[thm]{Definition}

\newtheorem{remark}[thm]{Remark}
\newtheorem{proposition}[thm]{Proposition}

\theoremstyle{definition}

\title[]{Continuous orbit equivalence rigidity for left-right wreath product actions}
\author{Yongle Jiang}

\address{Y. Jiang, School of Mathematical Sciences, Dalian University of Technology, Dalian, 116024, China}

\email{yonglejiang@dlut.edu.cn}

\begin{document}
\date{\today}

\begin{abstract}
Drimbe and Vaes proved an orbit equivalence superrigidity theorem for left-right wreath product actions in the measurable setting. We establish the counterpart result in the topological setting for continuous orbit equivalence. This gives us minimal, topologically free actions that are continuous orbit equivalence superrigid. One main ingredient for the proof is to show continuous cocycle superrigidity for certain generalized full shifts, extending our previous result with Chung. 
\end{abstract}



\maketitle

\section{Introduction}

In this paper, all our groups are assumed to be countable and discrete, and all our topological spaces are compact and Hausdorff. By a \textit{continuous action}, we mean an action of a group on a topological space by homeomorphisms. 

Continuous orbit equivalence (see Definition \ref{def: coe}) for general countable discrete group actions was formally introduced in Li's paper \cite{LX} about seven years ago. Roughly speaking, two continuous actions on compact Hausdorff spaces are said to be continuously orbit equivalent if we can identify their orbits in a continuous way. 
This is a weaker notion than topological conjugacy and
several connections between it and other topics, e.g. geometric group theory \cite{LX_agt} and C$^*$-algebras \cite{LX}, are quickly discovered by Li. 

By contrast, its counterpart in the measurable setting, i.e. orbit equivalence theory has a relatively long history, which emerges after the pioneering work of Dye \cite{dye}. During the last decades, much attention has been put on finding 
orbit equivalence superrigid actions, i.e. actions whose orbit equivalence classes consist only of themselves up to measurable conjugacy. Several impressive orbit equivalence superrigidity results have been discovered in \cites{Fur99a, Fur99b, MS02, Ki06,Io08,PV_adv,PV08,Ki09,PS09,Io14,TD14,CK15,Dr15,GITD16, Dr20, DIP, BTD, GH, HenH, HH1, HH2, HH3}.

Motivated by the success in finding orbit equivalence superrigid actions in the measurable setting, 
people also try to study the analogue in the topological setting. Until now, there are quite a few known rigidity results for continuous orbit equivalence, e.g. \cites{BT, SC95, LX, CM, cj, cohen, jiang_pams, jiang_ggd, GPS}. Nevertheless, in almost all of these results, either the source actions are not minimal as they contain fixed points or extra assumptions are put on the acting groups of the target actions. A possible exception is \cite[Corollary 6.5]{cj}, but we do not know whether it could be applied to minimal subshifts.
Hence, none of these actions are explicit minimal actions that are continuous orbit equivalence superrigid. Here, we say a continuous action $\Gamma\curvearrowright X$ is \textit{continuous orbit equivalence superrigid} if any topological free continuous action (see Definition \ref{def: various notions on t.d.s.}) of a group $\Lambda$ on a compact Hausdorff space $Y$ which is continuously orbit equivalent to it is actually topologically conjugate to it.

Recently, Drimbe and Vaes studied left-right wreath product group actions, more generally,  left-right wreath product equivalence relations $\mathcal{R}$ in the measurable setting \cite[Theorem E]{dv}. They were able to completely determine all essential free group actions $G\curvearrowright (Y,\eta)$ that are stably orbit equivalent to $\mathcal{R}$, i.e. the orbit equivalence relation $\mathcal{R}(G\curvearrowright Y)$ is isomorphic to the amplified equivalence relation $\mathcal{R}^t (t>0)$ in the sense of \cite[Definition 2.2]{Fur99b}, where the terminology of weak orbit equivalence was used.

Our goal in this paper is to extend the above mentioned result (with $t=1$) to the topological setting by considering continuous orbit equivalence. In particular, we present minimal, topologically free actions that are continuous orbit equivalence superrigid. More precisely, we prove the following result.   

\begin{thm}\label{thm: main thm for the whole paper}
Let $p$ be a prime number and $\Gamma_1$ be a finitely generated, non-torsion, non-amenable and icc group, e.g. $\Gamma_1=\mathbb{F}_n$, the free group with $n\geq 2$ generators. Let $\alpha$ be the left-right wreath product action $\frac{\Z}{p\Z}\wr_{\Gamma_1}(\Gamma_1\times \Gamma_1)\curvearrowright X:=(\frac{\Z}{p\Z})^{\Gamma_1}$ defined using the left translation $\frac{\Z}{p\Z}\curvearrowright \frac{\Z}{p\Z}$ and the left-right translation $\Gamma_1\times \Gamma_1\curvearrowright \Gamma_1$ (see Definition \ref{def: generalized wreath product actions} and \S  \ref{subsection: left-right translation and generalized full/Bernoulli shift  actions}). Then $\alpha$ is topologically free and minimal (see Definition \ref{def: various notions on t.d.s.}). Moreover, it is also a continuous orbit equivalence superrigid action, i.e. if it is continuously orbit equivalent to a topologically free action of a countable discrete group on a compact Hausdorff space, then the two actions are topologically conjugate.
\end{thm}

Recall that $\Gamma_1$ is called icc if every non-trivial element in it has an infinite conjugacy class. Towards proving such a theorem, a key step is to establish the continuous cocycle superrigidity theorem for generalized full shifts (see \S   \ref{subsection: left-right translation and generalized full/Bernoulli shift  actions}). Recall that in \cite{cj}, together with Chung, we proved a unified topological version of Popa's two celebrated cocycle superrigidity theorems \cite{popa_cocycle1, popa_cocycle2} for full shifts. However, Popa's theorems actually hold true for certain generalized Bernoulli shifts. Thus, the following result (see Corollary \ref{main corollary} (ii)) might be natural to expect.

\begin{thm}\label{thm: main typical thm for ccsr}
Let $\Gamma_1$ be a finitely generated, non-torsion and non-amenable group and $X_0$ be any finite set with more than one element. Then the generalized full shift action $\Gamma:=\Gamma_1\times \Gamma_1\curvearrowright X:=X_0^{\Gamma_1}$ defined using the left-right translation $\Gamma_1\times \Gamma_1\curvearrowright \Gamma_1$ is a continuous cocycle superrigid action, i.e. for any countable discrete group $G$, every continuous cocycle $c: \Gamma\times X\to G$ is trivial, i.e. it is  cohomologous to a group homomorphism $\phi: \Gamma
\to G$ via a continuous transfer map $b: X\to G$ (see \S \ref{subsection: continuous csr/oe superrigidity} for precise definition).  
\end{thm}

For a more general version, see Theorem \ref{thm: ccsr for generalized full shifts}. Let us discuss the proof of Theorem \ref{thm: main typical thm for ccsr} briefly.

At first glance, it seems routine to apply the method in \cite{cj} to deal with all generalized full shifts, but it is not completely clear to us how to do this, see the discussion in the last section in \cite{jiang_pams}. Putting the problem in a general context, observe that the generalized full shift $\Gamma\curvearrowright X_0^I$, where $\Gamma\curvearrowright I$ is a transitive action, may be treated as the coinduced action (see \cite[Definition 6.18]{KL_book}) from the trivial action of the stabilizer subgroup $Stab(i)\curvearrowright X_0$ to $\Gamma$, where $i\in I$ is any chosen point. 
By extending the method in \cite{cj}, we did prove a cocycle superrigidity theorem for certain coinduced actions in \cite{jiang_pams}, which could be applied to some generalized full shifts $\Gamma\curvearrowright X_0^I$, see e.g. Corollary \ref{main corollary} (i) for a concrete example. Unfortunately, the result in \cite{jiang_pams} could not be applied to the generalized full shifts in Theorem \ref{thm: main typical thm for ccsr}. The reason is that we assumed
$Stab(i)$ is a commensurated subgroup of $\Gamma$ in \cite{jiang_pams}. 
Recall that a subgroup $K$ is called \textit{commensurated} in $\Gamma$ if $gKg^{-1}\cap K$ has finite index in $K$ for all $g\in \Gamma$.
It is easy to check that for the left-right translation action $\Gamma=\Gamma_1\times \Gamma_1\curvearrowright I:=\Gamma_1$, this condition is never satisfied whenever $\Gamma_1$ is icc.
However, the icc assumption on $\Gamma_1$ is needed while proving Theorem \ref{thm: main thm for the whole paper} from Theorem \ref{thm: main typical thm for ccsr}.

To prove Theorem \ref{thm: main typical thm for ccsr}, our strategy can be described as follows. First, we observe that every continuous cocycle $c$ 
is cohomologous to a group homomorphism $\phi: \Gamma\to G$ via a \textit{measurable} transfer map $b: X\to G$ by applying Popa's cocycle superrigidity theorem. Then, we show that $b$ can be replaced by  a \textit{continuous} transfer map by modifying an argument used in the proof of the famous Liv\u{s}ic theorem. Let us briefly recall the classical Liv\u{s}ic theorem below.

For a compact Riemannian manifold $M$ with an Anosov diffeomorphism $T$ on it, Liv\u{s}ic \cite{liv1, liv2} proved a seminal result. It says that for any H\"older continuous map $f: M\to \R$, if there exists a measurable solution $u: M\to \R$ to the coboundary equation $f=u\circ T-u$ a.e. (with respect to a suitable measure $\mu$ on $M$), then there is a H\"older continuous solution $u'$, i.e. $f=u'\circ T-u'$ everywhere such that $u=u'$ $\mu$-a.e. Inspired by this result, people try to generalise it along two directions. One is to consider wider class of actions, e.g. partially hyperbolic actions, and the other one is to consider larger class of target groups than $\R$, see e.g. \cite{np, pp, pw, py, quas, SC}. Nevertheless, we are not aware of any work on extending this result to more general acting groups than $\mathbb{Z}$. Hence the combination of Popa's cocycle superrigidity theorem with a Liv\u{s}ic type argument used in our approach may be novel for dealing with non-cyclic acting groups.

The paper is organized as follows. In Section \ref{section: preliminaries}, we recall the definition of various notions used in this paper, including continuous orbit equivalence and related continuous cocycle superrigidity, generalized wreath product groups, left-right translation actions, generalized full/Bernoulli shifts and compression functions. In Section \ref{section: ccs for generalized full shifts}, we prove the continuous cocycle superrigidity for certain generalized full shifts, i.e. Theorem \ref{thm: ccsr for generalized full shifts}, and deduce Theorem \ref{thm: main typical thm for ccsr}  from Corollary \ref{main corollary} (ii). In Section \ref{section: coe type result for generalized wreath products}, we classify topological free actions that are continuously orbit equivalent to certain generalized wreath product actions, i.e. Theorem \ref{thm: main thm for generalized wreath product actions}, from which Theorem \ref{thm: main thm for the whole paper} follows easily. Finally, we ask several questions in Section \ref{section: ending questions}.

\section{Preliminaries}\label{section: preliminaries}

In this section, we recall various notions on groups and actions used in this paper.

\subsection{Generalized wreath product groups}

Let $\Gamma$ and $\Lambda$ be groups and $\Gamma\curvearrowright I$ be a group action on an index set $I$. The \textit{generalized wreath product group} $\Lambda\wr_I\Gamma$ is defined as the semi-direct product $(\oplus_I\Lambda)\rtimes \Gamma$, where $\Gamma$ acts on $\oplus_I\Lambda$ by the automorphisms $(\gamma \lambda)_i=\lambda_{\gamma^{-1}i}$, where $\gamma\in \Gamma$, $i\in I$ and $\lambda\in \oplus_I\Lambda$.
\subsection{Left-right translation and generalized full/Bernoulli shift actions}\label{subsection: left-right translation and generalized full/Bernoulli shift  actions}

Let $\Gamma$ be a group. The \textit{left-right translation action} $\Gamma\times \Gamma\curvearrowright \Gamma$ is given by $(s, t)g=sgt^{-1}$ for all $s, t, g\in \Gamma$.  In general, let $\Gamma\curvearrowright I$ be an action on an index set $I$ and $X_0$ be a finite set with at least two elements equipped with the discrete topology. The \textit{generalized full shift} $\Gamma\curvearrowright X_0^I$ is the continuous action given by $(\gamma x)_i=x_{\gamma^{-1}i}$, where $x=(x_i)_{i\in I}\in X_0^I$ and $\gamma\in \Gamma$. If $\Gamma\curvearrowright I$ is the left translation $\Gamma\curvearrowright \Gamma$, then this generalized full shift is just called a \textit{full shift}.
If the product space $X_0^I$ is equipped with a product measure $\mu_0^I$ defined on its Borel $\sigma$-algebra, where $\mu_0$ is a probability measure on $X_0$, then  $\Gamma\curvearrowright (X_0^I,\mu_0^I)$ is usually called a \textit{generalized Bernoulli shift}.

\subsection{Topologically free, minimal, transitive and topologically weakly mixing actions}

We need the following standard concepts on continuous actions, see e.g. Definition 7.1 and 7.15 in \cite{KL_book}.
\begin{definition}\label{def: various notions on t.d.s.}
Let $\Gamma\curvearrowright X$ be a continuous action. This action is 
\begin{enumerate}
\item \textit{topologically free} if 
for every $e\neq g\in \Gamma$, $\{x\in X:~gx\neq x\}$ is dense in $X$.
\item \textit{minimal} if $X$ has no non-empty proper $\Gamma$-invariant closed subset.
\item \textit{(topologically) transitive} if for all non-empty open sets $U, V\subset X$, there exists an $s\in \Gamma$ such that $sU\cap V\neq \emptyset$.
\item \textit{(topologically) weakly mixing} if the diagonal action $\Gamma\curvearrowright X\times X$ is (topologically) transitive, i.e. for all non-empty open sets $U_1$, $U_2$, $V_1$, $V_2\subseteq X$ there is an $s\in \Gamma$ such that $sU_1\cap V_1\neq \emptyset$ and $sU_2\cap V_2\neq \emptyset$.
\end{enumerate}
\end{definition}

Let us record the following lemma, i.e. Lemma 2.4 in \cite{PV_adv} for reference.

\begin{lem}\label{lem: Popa-Vaes lemma}
Let a group $\Gamma$ act on a set $I$. Then the following conditions are equivalent. 
\begin{itemize}
\item Every orbit of $\Gamma\curvearrowright I$ is infinite.
\item $\Gamma\curvearrowright I$ is weakly mixing, i.e. for every $A, B\subset I$ finite, there exists some $s\in \Gamma$ satisfying $sA\cap B=\emptyset$.
\end{itemize}
\end{lem}
 
Using this lemma, we obtain the following well-known fact. We include its proof for completeness.

\begin{lem}\label{lem: top.w.m. for generalized full shifts}
Let $X_0$ be a finite set with at least two elements and let $\Gamma\curvearrowright I$ be an action on an infinite set $I$ such that each orbit is infinite. Then the generalized full shift  $\Gamma\curvearrowright X:=X_0^I$ is topologically weakly mixing.
\end{lem}
\begin{proof}
Let $U_1$, $U_2$, $V_1$, $V_2\subseteq X$ be 
four non-empty open sets. By shrinking them if necessary, we may assume they are standard cylinder sets that are determined by four finite sets $I_1$, $I_2$, $J_1$ and $J_2$ in $I$ respectively. Then it suffices to show there exists some $s\in \Gamma$ such that $sI_1\cap J_1=\emptyset=sI_2\cap J_2$, or more directly, show $s(I_1\cup I_2)\cap (J_1\cup J_2)=\emptyset$. Since every orbit of $\Gamma\curvearrowright I$ is infinite, such an $s$ does exist by Lemma \ref{lem: Popa-Vaes lemma}. 
\end{proof}

\subsection{Continuous orbit equivalence}\label{subsection: coe}
The following notion  was introduced in \cite[Definition 2.5]{LX}.
\begin{definition}\label{def: coe}
Let $\alpha: G\curvearrowright X$ and $\beta: H\curvearrowright Y$ be two continuous actions. We say $\alpha$ is continuously orbit equivalent to $\beta$ if there is a homeomorphism $\phi:  X\rightarrow Y$ with inverse $\psi$ and continuous maps $c: G\times X\to H$ and $c': H\times Y\to G$ such that
\begin{align*}
\phi(\alpha_g(x))=\beta_{c(g, x)}(\phi(x)),\\
\psi(\beta_h(y))=\alpha_{c'(h, y)}(\psi(y))
\end{align*}
for all $g\in G$, $h\in H$, $x\in X$ and $y\in Y$.
\end{definition}
Here are a few facts on this notion.  If $\alpha$ and $\beta$ are topologically free, which is always the case in this paper, then $c$ and $c'$ are \textit{cocycles}, i.e. $c(g_1g_2, x)=c(g_1, g_2x)c(g_2, x)$ for all $g_1, g_2\in G$ and $x\in X$ (and a similar identity holds for $c'$) by \cite[Lemma 2.8]{LX}. Moreover, $G\ni g\mapsto c(g, x)\in H$ is a bijection for all $x\in X$ by \cite[Lemma 2.10]{LX}.
For later reference, we say $c'$ is the \textit{inverse cocycle} of $c$ and call such a coupling $(c, c', \phi)$ a \textit{continuous orbit equivalence coupling} (\textit{coe coupling} for short).
If cocycles $c$ and $c'$ do not depend on the space coordinates and they are considered as group isomorphisms between the two acting groups, then the above definition boils down to \textit{topological conjugacy}. 

Another equivalent way to characterize continuous orbit equivalence is to use transformation groupoids. For general references on groupoids  and their C$^*$-algebras, see \cites{Ren, SSW}.

For a continuous action $G\curvearrowright X$, the \textit{transformation groupoid} $G\ltimes X$ is given by the set $G\times X$ with multiplication $(g', x')(g, x)=(g'g, x)$ if $x'=gx$, inversion $(g, x)^{-1}=(g^{-1}, gx)$,
range map $r(g, x)=gx$ and source map $s(g, x)=x$. It is well-known that the reduced crossed product $C(X)\rtimes_rG$ is isomorphic to the reduced groupoid C$^*$-algebra $C^*_r(G\ltimes X)$ \cite[Example 9.3.8]{SSW}. Moreover, this isomorphism is induced by sending the canonical unitary $u_g\in C(X)\rtimes_rG$ (respectively, $f\in C(X)$) to the function  $G\ltimes X\ni (s, x)\mapsto \delta_{s, g}\in \mathbb{C}$ (respectively, $G\ltimes X\ni (s, x)\mapsto \delta_{e, s}f(x)\in\mathbb{C}$), where $\delta_{.}$ denotes the Dirac function, see e.g. \cite[Example 9.1.7 and 9.2.6]{SSW}.

The following theorem appeared as part of Theorem 1.2 in \cite{LX}.
\begin{thm}\label{thm: coe by isomorphic topological groupoids}
Let $\alpha:~ G\curvearrowright X$ and $\beta: ~H\curvearrowright Y$ be topologically free systems. The following are equivalent.
\begin{enumerate}
\item[(i)] $G\curvearrowright X$ is continuously orbit equivalent to $H\curvearrowright Y$.
\item[(ii)] $G\ltimes X\cong H\ltimes Y$ (as topological groupoids).
\end{enumerate}
\end{thm}
To show $(i)\Rightarrow (ii)$, Li proved that the map $G\ltimes X\ni (g, x)\mapsto (c(g, x), \phi(x))\in H\ltimes Y$ is an isomorphism, where $\phi: X\rightarrow Y$ denotes the homeomorphism witnessing the continuous orbit equivalence and $c: G\times X\to H$ denotes the associated cocycle. This explicit form will be needed in step 7 in the proof of Theorem \ref{thm: main thm for generalized wreath product actions}.

\subsection{Continuous cocycle superrigidity}\label{subsection: continuous csr/oe superrigidity}

In this paper, we say 
a continuous action $\Gamma\curvearrowright X$ on a compact Hausdorff space $X$ is a \textit{continuous cocycle superrigid} action if for any countable discrete group $G$, every continuous cocycle $c: \Gamma\times X\to G$ is trivial, i.e. there exists a group homomorphism $\phi: \Gamma\to G$ and a continuous map $b: X\to G$ such that $c(g, x)=b(gx)^{-1}\phi(g)b(x)$ holds for all $g\in G$ and $x\in X$. Once this equality holds, we say $c$ is \textit{cohomologous} to $\phi$ and call $b$ a \textit{transfer map}.

\subsection{Compression functions}
The following notion is frequently used when studying finitely generated groups, see e.g. \cite[Definition 2.1]{cj_etds}. 
\begin{definition}[Compression functions]\label{def: compression functions}
Let $G$ be a finitely generated group with a symmetric generating set $T$. Let $s$ be an element in $G$ with infinite order. Denote by $\ell_T$ the word length function on $G$, i.e. $\ell_T(g)=\min\{n:~\exists~t_1,\ldots, t_n\in T,~\text{s.t.}~
g=t_1\cdots t_n\}$. \textit{The compression function} $\rho_s$ is defined as  $\rho_s(x)=\min\{\ell_T(s^n):~n\geq x\}$, where $x\in \mathbb{R}_+$.
\end{definition}
By definition, we know $\rho_s(n)\leq \ell_T(s^n)$ for any positive integer $n$. For any $c\geq 0$, we denote the number $\sup\{\lambda>0:~\rho_s(\lambda)\leq c\}$ by $\rho_s^{-1}(c)$. This makes sense since $\rho_s$ is non-decreasing, goes to infinity and is constant on open intervals $(n, n+1)$ for all $n\in\mathbb{N}$, see \cite[Proposition 2.2]{cj_etds}.

\section{Continuous cocycle superrigidity for generalized full shifts}\label{section: ccs for generalized full shifts}

In this section, we show that certain generalized full shifts are continuous cocycle superrigid actions, which is a key ingredient for proving Theorem \ref{thm: main thm for the whole paper}.

\begin{thm}\label{thm: ccsr for generalized full shifts}
Let $\Gamma$ be a countable discrete group. Let $X_0$ be a finite set with at least two elements. Let $\Gamma\curvearrowright I$ be a transitive action satisfying the following conditions:
\begin{enumerate}
\item\label{condition 1 in ccsr thm} There exists some $s\in \Gamma$ such that $\lim_{n\to\infty}s^ni=\infty$ for all $i\in I$, i.e. for any finite subset $F\subset I$, there exists some $N\geq 1$ such that $s^ni\not\in F$ for all $n\geq N$.
\item\label{condition 2 in ccsr thm} There exists some increasing finite  set $F_n\subset I$ such that $I=\cup_nF_n$ and $E:=(\cup_{n\geq 0}s^{-n}F_n)\cap (\cup_{n\geq 0}s^nF_n)$ is finite, where $s$ is the element in (1).
\item\label{condition 3 in ccsr thm} The generalized Bernoulli shift  $\Gamma\curvearrowright (X_0^I,\mu_0^I)$, where $\mu_0$ is the uniform measure on $X_0$ with $\mu(\{a\})=\frac{1}{|X_0|}$ for every $a\in X_0$ and $\mu_0^I$ denotes the product measure, is a cocycle superrigid action in the following sense: for any countable discrete group $G$ and any measurable cocycle $c: \Gamma\times X_0^I\to G$, there exists some measurable map $b: X_0^I\to G$ and some group homomorphism $\phi: \Gamma\to G$ such that $c(g, x)=b(gx)^{-1}\phi(g)b(x)$ holds for all $g\in \Gamma$ and $\mu_0^I$-a.e. $x\in X_0^I$.
\end{enumerate}
 Then the generalized full shift $\Gamma\curvearrowright X_0^I$ defined using $\Gamma\curvearrowright I$ is a continuous cocycle superrigid action.
\end{thm}

To prove Theorem \ref{thm: ccsr for generalized full shifts}, the main idea is to mimic the proof of \cite[Theorem 8]{wp}, see also the proof of \cite[Theorem 1]{np}, \cite[Theorem 5.1]{pw} or \cite[Theorem 3.1]{SC}. Nevertheless, our proof needs new ingredient in order to deal with generalized full shifts of non-cyclic groups. A crucial ingredient is to replace the two degenerate cones with finite intersection, i.e. two half-lines in opposite directions in the Cayley graph of $\Z$ with $\cup_{n\geq 0}s^{-n}F_n$ and $\cup_{n\geq 0}s^nF_n$ for a well-chosen exhausting sequence $\{F_n\}$ for $I$ in the sense that condition \eqref{condition 2 in ccsr thm} holds.
\begin{remark}\label{remark: relation to our previous work}
If $\Gamma\curvearrowright I$ is the left translation action $\Gamma\curvearrowright \Gamma$, then initial results on the above theorem were proved in a preliminary version of \cite{cj}. But these results did not appear in the published version since they were covered by the main theorem in \cite{cj} after we discovered the one-end condition. The above condition \eqref{condition 2 in ccsr thm} is inspired by \cite[Lemma 3.4]{cj_etds}.
\end{remark}
Before proving Theorem \ref{thm: ccsr for generalized full shifts}, let us present concrete examples to which Theorem \ref{thm: ccsr for generalized full shifts} applies. 
\begin{cor}\label{main corollary}
Let $X_0$ be a finite set with at least two elements. Let $\Gamma\curvearrowright I$ be either one of the following actions:
\begin{itemize}
\item[(i)] $\Gamma_1\times \Gamma_2\curvearrowright \Gamma_1$ by $(\gamma_1, \gamma_2)\cdot g=\gamma_1\cdot g$ for all $\gamma_1, g\in \Gamma_1$ and $\gamma_2\in \Gamma_2$, where $\Gamma_2$ is a finitely generated group with one end, e.g. $\Gamma_2=\mathbb{Z}^n$ for $n\geq 2$.
\item[(ii)] $\Gamma_1\times \Gamma_1\curvearrowright \Gamma_1$ by left-right translation, where $\Gamma_1$ is a finitely generated, non-torsion and non-amenable group.
\item[(iii)] $SL_k(\mathbb{Z})\curvearrowright SL_k(\mathbb{Z})\cdot e_1$ by matrix left multiplication, where $k\geq 3$ and $e_1=(1,0,\ldots, 0)^t$.
\end{itemize}
Then the generalized full shift $\Gamma\curvearrowright X_0^I$ is a continuous cocycle superrigid action.
\end{cor}

\begin{proof}
Case (i) is actually a direct corollary of results in \cite{jiang_pams}. Since we do not use it in this paper, let us briefly sketch its proof. We may identify $I$ with $\Gamma/{\Gamma_2}$ and observe that the generalized full shift $\Gamma\curvearrowright X$ can be regarded as the coinduced action of the trivial action $\Gamma_2\curvearrowright X_0$ from the subgroup $\Gamma_2$ to the ambient group $\Gamma$, see \cite[Section 2]{jiang_pams}. Then we can apply Theorem 1.1 and Proposition 2.5 (8) in \cite{jiang_pams} to finish the proof of this case.

Next, let us check the three conditions in Theorem \ref{thm: ccsr for generalized full shifts} hold true for case (ii) and (iii). 

Consider case (ii) first. Let $g_0$ be an element in $\Gamma_1$ with infinite order. Set $s=(g_0, e)\in \Gamma$. Clearly, condition \eqref{condition 1 in ccsr thm} holds true. Moreover, condition \eqref{condition 3 in ccsr thm} holds by applying Popa's cocycle superrigidity theorem in the measurable setting \cite[Corollary 1.2]{popa_cocycle2} for non-amenable product groups since $\Gamma_1$ is non-amenable.

We are left to verify condition \eqref{condition 2 in ccsr thm} holds true.
 
For any $n\geq 1$, set $B(n):=\{g\in \Gamma:~\ell_T(g)\leq n\}$, where $T$ is any finite symmetric generating set for $\Gamma_1$. Define $F_n=[B(\frac{\rho_{g_0}(n)}{8}+1)\times B(\frac{\rho_{g_0}(n)}{8}+1)]\cdot e\subset I$. Clearly, $F_n\subseteq F_m$ if $n\leq m$ and $I=\cup_nF_n$. 
Let us show that $E=(\cup_{n\geq 0}s^{-n}F_n)\cap (\cup_{n\geq 0}s^nF_n)$ is finite. 

Take any $s^ni=s^{-m}j\in E$, where $n, m\geq 0$, $i\in F_n$ and $j\in F_m$. Then $s^{n+m}i=j$. Write $i=(g_n, h_n)\cdot e$ and $j=(g_m, h_m)\cdot e$, where 
$g_n, h_n\in B(\frac{\rho_{g_0}(n)}{8}+1)$ and $g_m, h_m\in B(\frac{\rho_{g_0}(m)}{8}+1)$.
Then, $g_0^{n+m}=g_mh_m^{-1}h_ng_n^{-1}$.

By Definition \ref{def: compression functions}, we have 
\begin{align*}
\rho_{g_0}(n+m)\leq \ell_T(g_0^{n+m})=\ell_T(g_mh_m^{-1}h_ng_n^{-1})&\leq 2(\frac{\rho_{g_0}(n)}{8}+1)+2(\frac{\rho_{g_0}(m)}{8}+1)\\
&\leq 4+\frac{\rho_{g_0}(n+m)}{2}.
\end{align*}
Therefore, $\rho_{g_0}(n+m)\leq 8$, i.e. $n+m\leq \rho_{g_0}^{-1}(8)$. 

Next, we estimate $\ell_{T'}(s^n (g_n,h_n))$, where $T'$ denotes the following generating set of $\Gamma$: 
\[T'=\{(t, e), (e, t): ~t\in T\}.\] 
Note that $\ell_{T'}(g, h)\leq \ell_T(g)+\ell_T(h)$ for any $(g, h)\in \Gamma$.

Let $g_n, h_n\in B(\frac{\rho_{g_0}(n)}{8}+1)$ be as above,
\begin{align*}
\ell_{T'}(s^n(g_n, h_n))=\ell_{T'}(g_0^ng_n, h_n)
&\leq \ell_T(g_0^ng_n)+\ell_T(h_n)\\
&\leq \ell_T(g_0)n+\ell_T(g_n)+\ell_T(h_n)\\
&\leq \ell_T(g_0)\rho_{g_0}^{-1}(8)+2+\frac{\rho_{g_0}(n+m)}{4}\\
&\leq \ell_T(g_0)\rho_{g_0}^{-1}(8)+4.
\end{align*}
Therefore, $s^ni=s^n(g_n, h_n)\cdot e\in B_{T'}(\ell_T(g_0)\rho_{g_0}^{-1}(8)+4)\cdot e$. Thus, 
\begin{align*}
E\subseteq B_{T'}(\ell_T(g_0)\rho_{g_0}^{-1}(8)+4)\cdot e.
\end{align*} 
This shows that $E$ is a finite set. Thus, the proof of case (ii) is finished.

Now, consider case (iii). Condition \eqref{condition 3 in ccsr thm} holds due to Popa's cocycle superrigidity theorem in the measurable setting \cite[Corollary 1.2]{popa_cocycle1} for infinite property (T) groups, e.g. $SL_k(\mathbb{Z})$ when $k\geq 3$. 
We are left to verify the first two conditions hold true.

Let $s$ be any element in $SL_k(\mathbb{Z})$ with $k$ many distinct eigenvalues $\{\lambda_i\}_{i=1}^k$ of non unit absolute value, say $|\lambda_j|>1$ iff $1\leq j\leq \ell$ for some $1\leq \ell\leq k-1$. The existence of such an $s$ will be proved in Proposition \ref{prop:existence of special matrix s}. Clearly, $s$ has infinite order.

Assume condition \eqref{condition 1 in ccsr thm} fails for the above $s$, then there exists some $i:=ge_1$, where $g\in SL_k(\mathbb{Z})$ and a sequence $n_i\to+\infty$ such that $s^{n_i}ge_1=ge_1$, i.e. $g^{-1}s^{n_i}ge_1=e_1$. This means that the first column of $g^{-1}s^{n_i}g$ is $e_1$, hence $g^{-1}s^{n_i}g$ must have eigenvalue 1, and thus $s$ has some eigenvalue with absolute value one, a contradiction.

Now, let us check condition \eqref{condition 2 in ccsr thm} holds true.

For any $n\geq 0$, define $F_n=\{v=(v_1,\ldots, v_k)^t\in I: ~\sup_{1\leq i\leq k}|v_k|\leq n\}\subset I$. Clearly, $F_n$ is finite, $I=\cup_nF_n$ and $F_n\subseteq F_m$ if $n\leq m$.

Take $g=(g_{ij})\in SL_k(\mathbb{Q})$ such that $gsg^{-1}=\mbox{Diag}(\lambda_1,\ldots, \lambda_k)$, the diagonal matrix with eigenvalues of $s$ as diagonal entries. 

It suffices to show that there exists some positive integer $N\geq 1$ such that if $s^nf=s^{-m}f'\in E$, where $f\in F_n$ and $f'\in F_m$, then $m+n\leq N$. 

First, let $c=\sup_{1\leq i, j\leq k}|g_{ij}|$. We may assume $g\in M_k(\frac{\Z}{L})$ for some positive integer $L$. For any $1\leq i\leq k$, let $g_i$ be the $i$-th row of $g$. Then observe that
\begin{align}\label{eq: vanishing matrix multiplication}
\text{For any column vector}~ v\in \mathbb{Z}^n,~ \text{if}~ |g_i\cdot v|\leq \frac{1}{2L}, ~\text{then} ~g_i\cdot v=0.
\end{align}
 Here, $g_i\cdot v$ is the usual matrix multiplication between the row vector $g_i$ and the column vector $v$.

From $s^nf=s^{-m}f'$, we know $s^{n+m}f=f'$, i.e. $\mbox{Diag}(\lambda_1^{m+n},\ldots, \lambda_k^{m+n})gf=gf'$.
Write $(x_1,\ldots, x_k)^t=gf\in g\cdot F_n$ and $(y_1,\ldots, y_k)^t=gf'\in g\cdot F_m$. So 
\begin{align}\label{eq: x_i and y_i are propotional}
\lambda_i^{m+n}x_i=y_i~ \text{for all $1\leq i\leq k$}.
\end{align}

Since $f\in F_n$ and $f'\in F_m$, the definition of $c$ and \eqref{eq: x_i and y_i are propotional} imply that $\sup_{1\leq i\leq k}|x_i|\leq ckn$ and $\sup_{1\leq i\leq k}|y_i|\leq ckm$.

If $1\leq j\leq \ell$, then $|\lambda_j|>1$. Note that $|x_j|=\frac{|y_j|}{|\lambda_j^{m+n}|}\leq \frac{ckm}{|\lambda_j|^{m+n}}\leq \frac{ck(m+n)}{|\lambda_j|^{m+n}}<\frac{1}{2L}$ if $m+n$ is large enough, say $m+n>M_j$ for some $M_j$.

Similarly, if $\ell<j\leq k$, then $|\lambda_j|<1$. We have
$|y_j|=|\lambda_j^{m+n}x_j|\leq |\lambda_j|^{m+n}ckn\leq |\lambda_j|^{m+n}ck(m+n)<\frac{1}{2L}$ if $m+n$ is large enough, say $m+n>N_j$.

Set $N=\sum_{1\leq j\leq \ell}M_j+\sum_{\ell<j\leq k}N_j$. Let us check this $N$ is what we want.

If $m+n>N$, then  $|x_j|<\frac{1}{2L}$ for all $1\leq j\leq \ell$ and $|y_j|<\frac{1}{2L}$ for all $\ell<j\leq k$. Since $x_j=g_j\cdot f$ and $y_j=g_j\cdot f'$, we deduce that $x_j=0$ for all $1\leq j\leq \ell$ and $y_j=0$ for all $\ell<j\leq k$ by \eqref{eq: vanishing matrix multiplication}. Thus we have $x_j=0=y_j$ for all $1\leq j\leq k$ by \eqref{eq: x_i and y_i are propotional}; equivalently, $f=0=f'$, which is a contradiction to the fact that $f, f'\in I=SL_k(\Z)\cdot e_1$.
\end{proof}
Now, let us show the existence of an $s$ used in the above proof.
\begin{proposition}\label{prop:existence of special matrix s}
Let $k\geq 2$. Then $SL_k(\mathbb{Z})$ contains an element $s$ with $k$-many distinct eigenvalues all having non unit absolute values.
\end{proposition}
\begin{proof}
For any $n\geq 1$, set $s_n=\begin{pmatrix}
1&n\\
1&n+1
\end{pmatrix}\in SL_2(\Z)$ and $t_n=\begin{pmatrix}
1&1&0\\
0&0&1\\
n+1&n&0
\end{pmatrix}\in SL_3(\Z)$.

A calculation shows that the characteristic polynomial of $s_n$ is $q_n(\lambda)=\lambda^2-(n+2)\lambda+1$ and its two eigenvalues are $\frac{n+2\pm \sqrt{n^2+4n}}{2}$. Clearly both eigenvalues have non unit absolute values. On the other hand, the characteristic polynomial of $t_n$ is $p_n(\lambda)=\lambda^3-\lambda^2-n\lambda-1$. Note that $\lim\limits_{\lambda\to-\infty}p_n(\lambda)=-\infty$, $\lim\limits_{\lambda\to+\infty}p_n(\lambda)=+\infty$, $p_n(-1)=n-3$, $p_n(0)=-1$ and $p_n(1)=-n-1$. If $n>3$, then we may apply the intermediate value theorem to $p_n(\lambda)$ and deduce that $t_n$ has three real eigenvalues, say $\{\lambda_i\}_{i=1}^3$, such that $\lambda_1<-1<\lambda_2<0<1<\lambda_3$. 

Next, we make the following claims.

Claim 1: $s_n$ and $s_m$ do not have common eigenvalues if $m>n\geq 1$.

Claim 2: $s_m$ and  $t_n$ do not have common eigenvalues if $m>n\geq 1$. 

To see claim 1 holds, just observe that if $n\neq m$, then $q_n(\lambda)=0=q_m(\lambda)$ implies $\lambda=0$, which is clearly not an eigenvalue of $s_n$.

To prove claim 2, let us assume $q_m(\lambda)=0=p_n(\lambda)$ holds for some $\lambda$ and some $m>n$. Then a calculation shows $\lambda=\frac{m+2}{(m+2)(m+1)-(n+1)}$. By plugging this value of $\lambda$ into $q_m(\lambda)=0$, we get the identity $(n+1)^2=m(m+2)(n+1)+m(m+2)^2$. Clearly, this is absurd since $m>n\geq 1$. This finishes the proof of claim 2.

Using these two claims, we can easily construct an $s$ with the required property.

If $k$ is even, then we take $\frac{k}{2}$-many distinct positive integers $n_i>3$. Thus, $s_{n_i}$'s have no common eigenvalues by claim 1. Define $s=\mbox{Diag}(s_{n_1},\ldots, s_{n_{\frac{k}{2}}})$, the block diagonal matrix with $s_{n_i}$'s on the diagonal. Clearly, this $s$ has $k$-many distinct eigenvalues.

If $k$ is odd, then we just pick $\frac{k-1}{2}$-many distinct positive integers $n_i>3$ with $n_{\frac{k-1}{2}}$ being the smallest one among them. Then these $\frac{k-1}{2}$-many matrices $s_{n_i}$, where $1\leq i\leq \frac{k-3}{2}$ and $t_{n_{\frac{k-1}{2}}}$ have exactly $k$-many distinct eigenvalues in total by claim 1 and claim 2. Hence we may define $s=\mbox{Diag}(s_{n_1},\ldots, s_{n_{\frac{k-3}{2}}}, t_{n_{\frac{k-1}{2}}})$.
\end{proof}

Let $\Gamma\curvearrowright I$ be a transitive action on a countably infinite set $I$ and $\Gamma\curvearrowright X:=X_0^I$ be the generalized full shift, where $X_0$ denotes a finite set with $|X_0|\geq 2$. Let $d_0$ be a compatible metric on $X_0$ with upper bound 1. Then we may define a compatible metric $d$ on $X$ using $d_0$. For example, we may identify $I$ with $\mathbb{N}$ and define $d(x, y)=\sum_{i\in\mathbb{N}}\frac{1}{2^i}d_0(x_i,y_i)$.

Recall that a main tool to study continuous cocycle superrigidity for full shifts in \cite{cj} is the following notion of homoclinic equivalence relations:
\[\Delta_X:=\{(x, y)\in X\times X:~\lim\limits_{g\to\infty}d(gx, gy)=0\}.\]
To study the generalized full shifts, we need to introduce finer equivalence relations.

Fix any increasing finite subsets $F_n\subset I$ such that $I=\cup_nF_n$. If $s\in \Gamma$ has infinite order, then we may define the following equivalence relations:
\begin{align*}
\Delta_{X, s}^+&:=\{(x, y)\in X\times X:~\lim\limits_{n\to +\infty}\sup_{i\in F}d_0(x_{s^{-n}i}, y_{s^{-n}i})=0~\text{for any finite set $F\subset I$}\},\\
\Delta_{X, s}^-&:=\{(x, y)\in X\times X:~\lim\limits_{n\to +\infty}\sup_{i\in F}d_0(x_{s^ni}, y_{s^ni})=0~\text{for any finite set $F\subset I$}\},\\
\Theta_{X, s}^+&:=\{(x, y)\in X\times X:~\lim\limits_{n\to +\infty}\sup_{i\in F_n}d_0(x_{s^{-n}i}, y_{s^{-n}i})=0\},\\
\Theta_{X, s}^-&:=\{(x, y)\in X\times X:~\lim\limits_{n\to +\infty}\sup_{i\in F_n}d_0(x_{s^ni}, y_{s^ni})=0\}.
\end{align*}
Note that $\Delta_{X, s}^-=\Delta_{X, s^{-1}}^+$, $\Theta_{X, s}^-=\Theta_{X, s^{-1}}^+$, $\Delta_{X, s}^+\supset\Theta_{X, s}^+$ and $\Delta_{X, s}^-\supset\Theta_{X, s}^-$ by our assumption on $F_n$.

We are ready to prove Theorem \ref{thm: ccsr for generalized full shifts}.
\begin{proof}[Proof of Theorem \ref{thm: ccsr for generalized full shifts}]

Let $c: \Gamma\times X_0^I\to G$ be a continuous cocycle into a countable discrete group $G$. We aim to show $c$ is cohomologous to a group homomorphism.

Fix $s\in \Gamma$ and $F_n\subset I$ satisfying condition \eqref{condition 1 in ccsr thm} and \eqref{condition 2 in ccsr thm} respectively.
First, observe that since $c$ is continuous and $G$ is discrete, we may find a finite $F\subset I$ such that the value of $c(s, -)$ depends only on the $F$-coordinates of the second entry, i.e. if $x, y\in X$ with $x_f=y_f$ for all $f\in F$, then $c(s, x)=c(s, y)$. By condition \eqref{condition 2 in ccsr thm}, we may further pick some integer $M\geq 1$ such that $F\subset F_n$ for all $n>M$. 

We split the proof into proving several claims.

\textbf{Claim 1}: Let $(x, y)\in \Delta_{X, s}^+$ (respectively $\Delta_{X, s}^-$). Then $\lim_{n\to +\infty}c(s^n, x)^{-1}c(s^n, y)$ (respectively, $\lim_{n\to-\infty}c(s^n, x)^{-1}c(s^n, y)$) is well-defined.

\begin{proof}

Let us show $\lim_{n\to\infty}c(s^n, x)^{-1}c(s^n, y)$ is well-defined below. The second limit can be treated exactly the same.

It suffices to show that there exists $N\geq 1$ such that 
\[c(s^n, x)^{-1}c(s^n, y)=c(s^m, x)^{-1}c(s^m, y),~\forall~m>n\geq N.\]

Equivalently, we need to show 
\[c(s^m, x)c(s^n, x)^{-1}=c(s^m, y)c(s^n, y)^{-1}, \]
i.e. 
\[c(s^{m-n}, s^nx)=c(s^{m-n}, s^ny),~\forall~m>n\geq  N.\]
Notice that the cocycle relation implies
\begin{align*}
c(s^{m-n}, s^nx)=\prod_{i=m-n-1}^0c(s, s^{i+n}x).
\end{align*}
It suffices to show that for each $0\leq i\leq m-n-1$, we have 
\[c(s, s^{i+n}x)=c(s, s^{i+n}y).\]
By the definition of $F$, we just need to make sure that for each $0\leq i\leq m-n-1$, $(s^{i+n}x)_j=(s^{i+n}y)_j$ for all $j\in F$, i.e. $x_{s^{-n-i}j}=y_{s^{-n-i}j}$ for all $j\in F$.
Since $(x, y)\in \Delta_{X, s}^+$ and $X_0$ is finite, the above clearly holds true for some large $N$ and all $m>n\geq N$.
\end{proof}

%To guarantee the above holds, we just need to make sure for all $0\leq i\leq m-n-1$ and all $j\in F$, $s^{-n-i}j\in (\Gamma\setminus T)i_0$. Notice that $-m-1\leq -n-i\leq -n\leq -N$, by assumption on $\Gamma\curvearrowright I$, we may take $N$ large enough to make sure this holds.

Following the notation in \cite[Subsection 2.3]{cj}, we write  
\begin{align*}
c_f^{(s), +}(x, y)&=\lim_{n\to +\infty}c(s^n, x)^{-1}c(s^n, y),~\forall~(x, y)\in \Delta_{X, s}^+,\\
c_f^{(s), -}(x, y)&=\lim_{n\to-\infty}c(s^n, x)^{-1}c(s^n, y), ~\forall~(x, y)\in\Delta_{X, s}^-.
\end{align*}

Set $\mu=\mu_0^I$. Treating the continuous cocycle $c$ as a $\mu$-measurable one and applying condition \eqref{condition 3 in ccsr thm}, we may find some measurable map $b: X\to G$ and a group homomorphism $\phi: \Gamma\to G$ such that 
\[c(g, x)=b(gx)^{-1}\phi(g)b(x),~\forall~\mu\mbox{-a.e.}~x\in X,~\forall~g\in \Gamma. \]

Our goal is to show that $b$ equals a continuous map $b': X\to G$ for $\mu$-a.e. $x\in X$. %Once this holds, then $c(g, x)$ equals $b'(gx)^{-1}\phi(g)b(x)$ for $\mu$-a.e. $x\in X$ and all $g\in G$. Since both sides are continuous functions, the equality holds for all $x\in X$ and we are done.

\textbf{Claim 2}: There exists some conull Borel set $X'\subset X$ such that
$c_f^{(s), +}(x, y)=b(x)^{-1}b(y)$ if $(x, y)\in \Delta_{X, s}^+\cap (X'\times X')$; similarly, $b(x)^{-1}b(y)=c_f^{(s), -}(x, y)$ if $(x, y)\in \Delta_{X, s}^-\cap (X'\times X').$

\begin{proof}

Let $D=\{x\in X: ~c(g, x)=b(gx)^{-1}\phi(g)b(x)~\text{for all}~g\in\Gamma\}$. Note that $D$ is $\Gamma$-invariant by cocycle relation and $\mu(D)=1$ since $\Gamma$ is countable. By Lusin's theorem \cite[Theorem 2.24]{Rudin}, there is a compact set $C\subset D$ with $\mu(C)>\frac{1}{2}$ such that the restriction $b|_C$ is continuous.
Define the following set \[M_C:=\{x\in D:~\lim\limits_{n\to+\infty}\frac{1}{n}\sum_{i=1}^n\chi_C(s^ix)= \mu(C)~\mbox{and}~\lim\limits_{n\to +\infty}\frac{1}{n}\sum_{i=1}^n\chi_C(s^{-i}x)= \mu(C)\}.\]
Here, $\chi_C$ denotes the indicator function on $C$.
By condition \eqref{condition 1 in ccsr thm}, the subaction $\mathbb{Z}=\langle s\rangle\curvearrowright (X,\mu)$ is mixing and hence ergodic. Indeed, this can be checked easily by noticing that for any non-empty finite set $A, B\subset I$, there is some $N\geq 1$ such that $s^nA\cap B=\emptyset$ for all $n$ with $|n|\geq N$. Then by applying
Birkhoff's pointwise ergodic theorem \cite[Theorem 4.28]{KL_book} to this subaction, we deduce  that
$\mu(M_C)=1$. 

Define $X'=M_C$. For any $(x, y)\in D\times D$, we have
\begin{align}\label{eq: kill middle b(..) by taking limits}
c(s^m, x)^{-1}c(s^m, y)=b(x)^{-1}\phi(s)^{-m}b(s^mx)b(s^my)^{-1}\phi(s)^mb(y),~\forall~m\in\mathbb{Z}.
\end{align}
Let $(x, y)\in \Delta_{X, s}^+\cap (X'\times X')$. 
Since $\mu(C)>\frac{1}{2}$, the definition of $X'$ shows there is a sequence $n_i\to +\infty$ with $s^{n_i}x, s^{n_i}y\in C$. To see this, let $A_n=\{1\leq i\leq n:~x\in s^{-i}C\}$ and $B_n=\{1\leq i\leq n:~y\in s^{-i}C\}$. Since $x, y\in X'$, we know $\lim\limits_{n\to+\infty}\frac{|A_n|}{n}=\mu(C)=\lim\limits_{n\to+\infty}\frac{|B_n|}{n}$. Hence,
\begin{align*}
\liminf_{n\to+\infty}\frac{|A_n\cap B_n|}{n}&=\liminf_{n\to+\infty}\frac{|A_n|+|B_n|-|A_n\cup B_n|}{n}\\
&\geq \liminf_{n\to+\infty}\frac{|A_n|+|B_n|-n}{n}=2\mu(C)-1>0.
\end{align*}
Now, it is clear that the desired sequence $n_i$ exists.

Next, observe that $(x, y)\in\Delta_{X, s}^+$ implies $d(s^{n_i}x, s^{n_i}y)\to 0$ as $n_i\to +\infty$. Set $m=n_i$ and take limit on both sides of \eqref{eq: kill middle b(..) by taking limits}, we deduce that $c_f^{(s), +}(x, y)=b(x)^{-1}b(y)$ from claim 1 and the fact that $b|_C$ is continuous. The second half of claim 2 can be proved similarly.
\end{proof}

Before making the next claim, recall that at the beginning of the proof, we have found a finite subset $F\subset I$ and some positive integer $M$ such that $c(s, -)$ only depends on the $F$-coordinates of the second entry and $F\subset F_i$ for all $i>M$. Note that this implies $s^{-i}F\subset s^{-i}F_i$ for all $i>M$.

\textbf{Claim 3:} 
Let $(x, y)\in \Delta_{X, s}^+\cap (X'\times X')$. 
Assume that either $x_k=y_k$ for all $k\in \cup_{i\geq 0}s^{-i}F$ or
$x$ is sufficiently close to $y$, then $b(x)=b(y)$. Similarly, let $(x, y)\in \Delta_{X, s}^-\cap (X'\times X')$. Assume that either $x_k=y_k$ for all $k\in \cup_{i\geq 0}s^iF$ or $x$ is sufficiently close to $y$, then $b(x)=b(y)$.

\begin{proof}
Let $(x, y)\in \Delta_{X, s}^+\cap (X'\times X')$. By claim 2, it suffices to check that $c_f^{(s), +}(x, y)=e$ under either of the two assumptions.

Recall that $c_f^{(s), +}(x, y)=\lim\limits_{n\to+\infty}c(s^n, x)^{-1}c(s^n, y)$ and $c(s^n, x)=\prod_{i=n-1}^0c(s, s^ix)$ for all $n\geq 1$ by cocycle relation. It suffices to show $c(s, s^ix)=c(s, s^iy)$ for all $i\geq 0$. 
Since $c(s, -)$ only depends on the $F$-coordinates of the second entry, one just needs to make sure $x_{s^{-i}f}=y_{s^{-i}f}$ for all $i\geq 0$ and $f\in F$, i.e. $x$ and $y$ take the same value on all coordinates in $\cup_{i\geq 0}s^{-i}F$, which is exactly the first assumption.

Next, from $(x, y)\in\Delta_{X, s}^+$, we deduce that $\lim\limits_{i\to\infty}\sup_{f\in F}d_0(x_{s^{-i}f}, y_{s^{-i}f})=0$. This implies that there exists some $N_1>0$ such that for all $i>N_1$, we have $x_{s^{-i}f}=y_{s^{-i}f}$ for all $f\in F$, i.e. $x$ and $y$ take the same value on all coordinates in $\cup_{i>N_1}s^{-i}F$.

Therefore, if $x$ is sufficiently close to $y$ such that $x_{s^{-i}f}=y_{s^{-i}f}$ for all $f\in F$ and $0\leq i\leq N_1$, then we also have $c_f^{(s), +}(x, y)=e$; equivalently, $b(x)=b(y)$. 

The case $(x, y)\in \Delta_{X, s}^-\cap (X'\times X')$ can be handled similarly.
\end{proof}

Recall that $F_n\subset I$ satisfies condition \eqref{condition 2 in ccsr thm} in the statement of the theorem, $M$ and $F$ are chosen to be finite. Hence 
\[E':=[(\cup_{n\geq 0}s^{-n}F_n)\cap (\cup_{n\geq 0}s^nF_n)]\cup(\cup_{n=-M}^Ms^nF)~\text{is also finite}.\] 
Let $v\in X_0^{E'}$ and $[v]:=\{x\in X~:x_f=v_f,~\forall~f\in E'\}$. Define $p_v: [v]\times [v]\to [v]$ as follows:
\begin{align}\label{eq: def of p_v(x, y)}
(p_v(x, y))_i=\begin{cases}
y_i,~&~\mbox{if}~i\in (\cup_{n\geq 0}s^{-n}F_n)\cup(\cup_{n=-M}^Ms^nF),\\
x_i,~&~\mbox{elsewhere}.
\end{cases}
\end{align}

\textbf{Claim 4:} For each $v\in X_0^{E'}$, $(p_v)_*(\mu\times \mu)$ is absolutely continuous with respect to $\mu|_{[v]}$, i.e. $(p_v)_*(\mu\times \mu)\ll \mu|_{[v]}$. 
\begin{proof}
We show below that the two measures are in fact equivalent. Note that the finiteness of $E'$ rather than its explicit definition is needed for the proof. 

To prove claim 4, we argue similarly as in the proof of Lemma 1 in \cite[Section 5]{wp}. It suffices to find a constant $\tilde{D}>0$ such that for each cylinder set $A_{K, v'}=\{x\in X:~x_k=v'_k,~\forall~k\in K\}$, we have $(p_v)_*(\mu\times \mu)(A_{K, v'})\leq \tilde{D}\cdot \mu|_{[v]}(A_{K, v'})$. Here, $K\subset I$ is any non-empty finite set and $v'\in X_0^{K}$. 
For ease of notations, set 
\[I_+=\cup_{n\geq 0}s^nF_n,~ I_-=\cup_{n\geq 0}s^{-n}F_n,~\text{and}~I_0=\cup_{n=-M}^Ms^nF.\]
On the one hand, we have
\begin{align}\label{eq: description of preimage of cylinder sets}
\begin{split}
&\quad p_v^{-1}(A_{K, v'})\\
&=\{(x, y)\in [v]\times [v]:~p_v(x, y)\in A_{K, v'}\}\\
&=\{(x, y)\in [v]\times [v]:~(p_v(x, y))_i=v'_i,~\forall~i\in K\}\\
&=\left\lbrace {(x, y)\in [v]\times [v]:~y_i=v'_i,~\forall~i\in K\cap (I_-\cup I_0)};~{x_i=v'_i,~\forall~i\in K\setminus (I_-\cup I_0)}\right\rbrace\\
&=\left\lbrace {(x, y)\in X\times X:y_i=v'_i,~\forall~i\in K\cap (I_-\cup I_0)},~ y_i=v_i,~\forall~i\in E';~\right.\\
&\quad\left. {~x_i=v'_i,~\forall~i\in  K\setminus (I_-\cup I_0)},~ x_i=v_i,~\forall~i\in E'\right\rbrace\\
&=\left\lbrace x\in X:~x_i=v'_i,~\forall~i\in  K\setminus (I_-\cup I_0),~ x_i=v_i,~\forall~i\in E'\right\rbrace\times\\
&\quad \left\lbrace y\in X:~y_i=v'_i,~\forall~i\in K\cap (I_-\cup I_0),~ y_i=v_i,~\forall~i\in E'\right\rbrace.
\end{split}
\end{align}
On the other hand, 
\begin{align*}
\mu|_{[v]}(A_{K, v'})=\mu(A_{K, v'}\cap [v])=\begin{cases}
\frac{1}{|X_0|^{|E'\cup K|}},~&\text{if $v'_i=v_i$ for all $i\in E'\cap K$},\\
0,~&\text{otherwise}.
\end{cases}
\end{align*}
Here, we make the convention that if $E'\cap K=\emptyset$, then $\mu(A_{K, v'}\cap [v])=\frac{1}{|X_0|^{|E'\cup K|}}$.
Note that since $E'\cap K=[E'\cap (K\setminus (I_-\cup I_0))]\sqcup [E'\cap (K\cap (I_-\cup I_0))]$, it is clear that if $v'_i\neq v_i$ for some $i\in E'\cap K$, then $\mu|_{[v]}(A_{K, v'})=0=(p_v)_*(\mu\times \mu)(A_{K, v'})$ by \eqref{eq: description of preimage of cylinder sets}. Hence without loss of generality, we may assume $v'_i=v_i$ for all $i\in E'\cap K$ from now on. This means that $\mu|_{[v]}(A_{K, v'})=\frac{1}{|X_0|^{|E'\cup K|}}$.

Therefore, from \eqref{eq: description of preimage of cylinder sets}, we deduce that
\begin{align*}
&\quad(p_v)_*(\mu\times \mu)(A_{K, v'})\\
&=(\mu\times \mu)(p_v^{-1}(A_{K, v'}))\\
&=\mu(\{x\in X:~x_i=v'_i,~\forall~i\in E'\cup [K\setminus (I_-\cup I_0)]\})\cdot\\
&\quad\mu(\{y\in X:~y_i=v'_i,~\forall~i\in E'\cup(K\cap (I_-\cup I_0))\})\\
&=\frac{1}{|X_0|^{|E'\cup [K\setminus (I_-\cup I_0)]|+|E'\cup(K\cap (I_-\cup I_0))|}}\\
&=\frac{1}{|X_0|^{|E'\cup [K\setminus (I_-\cup I_0)]|+|(K\cap (I_-\cup I_0))\setminus E'|+|E'|}}\\
&=\frac{1}{|X_0|^{|K\cup E'|+|E'|}}~(\text{since $[E'\cup (K\setminus (I_-\cup I_0))]\sqcup [(K\cap (I_-\cup I_0))\setminus E']=K\cup E'$})\\
&=\frac{1}{|X_0|^{|E'|}}\cdot \frac{1}{|X_0|^{|K\cup E'|}}
=\frac{1}{|X_0|^{|E'|}}\cdot \mu|_{[v]}(A_{K, v'}).
\end{align*}
Thus, by setting $\tilde{D}=\frac{1}{|X_0|^{|E'|}}$, we have shown that $(p_v)_*(\mu\times \mu)=\tilde{D}\cdot \mu|_{[v]}$. In particular, $(p_v)_*(\mu\times \mu)\ll \mu|_{[v]}$.
\end{proof}

Since $\mu(X')=1$, we know $X'\cap [v]$ is conull in $[v]$, i.e. $\mu|_{[v]}([v]\setminus (X'\cap [v]))=0$. Hence by the above absolute continuity, we deduce that $(p_v)_*(\mu\times \mu)([v]\setminus (X'\cap [v]))=0$, i.e. $0=(\mu\times \mu)(p_v^{-1}([v]\setminus (X'\cap [v])))=(\mu\times \mu)(([v]\times [v])\setminus p_v^{-1}(X'\cap [v]))$. Thus,
$p_v^{-1}(X'\cap [v])$ has the same $\mu\times \mu$ measure as $[v]\times [v]$. 
For each $y\in [v]$, define $B_y^v=\{x\in [v]:~(x, y)\in p_v^{-1}(X'\cap [v])\}$. Then $\mu(B^v_y)=\mu([v])$ for $\mu$-a.e. $y\in [v]$ by Fubini's theorem \cite[Theorem 8.8]{Rudin}.

Let $X''=X'\cap (\cup_{v\in X_0^{E'}}\{y\in [v]:~\mu(B^v_y)=\mu([v])\})$. Then $\mu(X'')=1$ since $\mu(X')=1$ and $\sqcup_{v\in X_0^{E'}}[v]=X$.

\textbf{Claim 5}: Let $(x, y)\in X''\times X''$. If $x_f=y_f$ for all $f\in E'$, then $b(x)=b(y)$.

\begin{proof}
Let $(x, y)\in X''\times X''$  with $x_f=y_f$ for all $f\in E'$. We denote by $v=x|_{E'}\in X_0^{E'}$ the restriction of $x$ to the $E'$-variables. By the definition of $X''$, we have $x, y\in [v]$ and $\mu(B_x^v)=\mu(B_y^v)=\mu([v])$. Hence $\mu([v]\cap B_x^v)=\mu([v]\cap B_y^v)=\mu([v])$. Therefore, $\mu([v]\cap B_x^v\cap B_y^v)=\mu([v])>0$.
Pick any $z\in [v]\cap B_x^v\cap B_y^v$. Then
$p_v(z, x), p_v(z, y)\in [v]\cap X'$.

Next, let us check that $(x, p_v(z, x)),(p_v(z, y), y)\in \Theta_{X, s}^+$ and $(p_v(z, x), p_v(z, y))\in \Theta_{X, s}^-$, where $\Theta_{X, s}^{\pm}$ are both defined using $F_n$ satisfying condition \eqref{condition 2 in ccsr thm}.

By the definition of $\Theta_{X, s}^{\pm}$, we need to check the following hold.
\begin{align}
\lim\limits_{n\to+\infty}\sup\limits_{i\in F_n}d_0(x_{s^{-n}i}, (p_v(z, x))_{s^{-n}i})=0,\label{eq: identity 1 in claim 4 in ccsr thm}\\
\lim\limits_{n\to+\infty}\sup\limits_{i\in F_n}d_0((p_v(z, x))_{s^ni}, (p_v(z, y))_{s^ni})=0,\label{eq: identity 2 in claim 4 in ccsr thm}\\
\lim\limits_{n\to+\infty}\sup\limits_{i\in F_n}d_0((p_v(z, y))_{s^{-n}i}, y_{s^{-n}i})=0.\label{eq: identity 3 in claim 4 in ccsr thm}
\end{align}
By \eqref{eq: def of p_v(x, y)}, we know $(p_v(z, x))_{s^{-n}i}=x_{s^{-n}i}$ for all $i\in F_n$, hence \eqref{eq: identity 1 in claim 4 in ccsr thm} holds. Similarly, we may also verify \eqref{eq: identity 3 in claim 4 in ccsr thm}.

To check \eqref{eq: identity 2 in claim 4 in ccsr thm} holds, fix any $n\geq 0$ and $i\in F_n$. If $s^ni\in (\cup_{n\geq 0}s^{-n}F_n)\cup (\cup_{n=-M}^Ms^nF)$, then 
$(p_v(z, x))_{s^ni}\overset{\eqref{eq: def of p_v(x, y)}}{=}x_{s^ni}=y_{s^ni}\overset{\eqref{eq: def of p_v(x, y)}}{=}(p_v(z, y))_{s^ni}$. Indeed, the second equality holds since we have $s^ni \in (\cup_{n\geq 0}s^{-n}F_n)\cup (\cup_{n=-M}^Ms^nF)\subset E'$ and we have assumed that $x|_{E'}=y|_{E'}$. If $s^ni\not\in (\cup_{n\geq 0}s^{-n}F_n)\cup (\cup_{n=-M}^Ms^nF)$, then \eqref{eq: def of p_v(x, y)} tells us that
$(p_v(z, x))_{s^ni}=z_{s^ni}=(p_v(z, y))_{s^ni}$. Hence in both cases, we always have $p_v(z, x)_{s^ni}=(p_v(z, y))_{s^ni}$ for all $n\geq 0$ and $i\in F_n$, therefore \eqref{eq: identity 2 in claim 4 in ccsr thm} also holds.

Clearly, we may write
\begin{align*}
b(x)^{-1}b(y)=[b(x)^{-1}b(p_v(z, x))]\cdot [b(p_v(z, x))^{-1}b(p_v(z, y))]\cdot [b(p_v(z, y))^{-1}b(y)].
\end{align*}
To show $b(x)=b(y)$, it suffices to show that
\begin{align}\label{eq: triple identity: b()=b()=b()}
b(x)=b(p_v(z, x))~, ~b(p_v(z, x))=b(p_v(z, y))~\text{and}~ b(p_v(z, y))=b(y). 
\end{align}

Since $\Theta_{X, s}^+\subseteq \Delta_{X, s}^+$ (respectively, $\Theta_{X, s}^-\subseteq \Delta_{X, s}^-$) and $X''\subseteq X'$, we may apply claim 3 to the three pairs of points: $(x, p_v(z, x))$, $(p_v(z, x), y)$ and $(p_v(z, x), p_v(z, y))$. This means that to show \eqref{eq: triple identity: b()=b()=b()}, it suffices to check the following hold.
\begin{align}
x_j&=(p_v(z,x))_j,~\forall~j\in \cup_{i\geq 0}s^{-i}F,\label{eq: identity 4 in claim 4 in ccsr}\\
(p_v(z, x))_j&=(p_v(z, y))_j,~\forall~j\in \cup_{i\geq 0}s^iF,\label{eq: identity 5 in claim 4 in ccsr}\\
(p_v(z, y))_j&=y_j,~\forall~j\in \cup_{i\geq 0}s^{-i}F.\label{eq: identity 6 in claim 4 in ccsr}
\end{align}
Recall that $F\subset F_i$ for all $i>M$ and hence $s^{-i}F\subseteq s^{-i}F_i$ and $s^iF\subseteq s^iF_i$ for all $i>M$.
Moreover, when checking \eqref{eq: identity 1 in claim 4 in ccsr thm} holds, we have shown $x_j=(p_v(z, x))_j$ for all $j\in \cup_{i\geq 0}s^{-i}F_i$. Hence, it suffices to check \eqref{eq: identity 4 in claim 4 in ccsr} for every $j\in \cup_{i=0}^Ms^{-i}F\subset \cup_{i=-M}^Ms^iF$. For this $j$, we have $(p_v(z, x))_j=x_j$ by definition \eqref{eq: def of p_v(x, y)}. Hence \eqref{eq: identity 4 in claim 4 in ccsr} is verified. Similarly, we may check \eqref{eq: identity 6 in claim 4 in ccsr} holds.

To check \eqref{eq: identity 6 in claim 4 in ccsr}, recall that when checking \eqref{eq: identity 2 in claim 4 in ccsr thm} holds, we have shown $(p_v(z, x))_j=(p_v(z, y))_j$ for all $j\in \cup_{i\geq 0}s^iF_i$. So we may assume now $j\in \cup_{i=0}^Ms^iF\subseteq \cup_{i={-M}}^Ms^iF\subset E'$. Then for this $j$, we have $(p_v(z, x))_j=x_j=y_j=(p_v(z, y))_j$. Thus \eqref{eq: identity 5 in claim 4 in ccsr} holds.

To sum up, we have shown that if $x$ is sufficiently close to $y$ such that $x_f=y_f$ for all $f\in E'$, then $b(x)=b(y)$.
\end{proof}

\textbf{Claim 6}: There exists a continuous map $b': X\to G$ such that $b'(x)=b(x)$ for $\mu$-a.e. $x\in X$.
\begin{proof}
Set $X'''=\cap_{\gamma\in \Gamma}\gamma (X'')\subseteq X''$. Since $\Gamma$ is countable and $\mu(X'')=1$, we know that $\mu(X''')=1$. Thus $X'''$ is dense in $X$ since $\mu(U)>0$ for any non-empty open set $U\subseteq X$. Fix any $x\in X$, pick a sequence $x_i\in X'''$ such that $\lim\limits_{i\to+\infty}x_i=x$, then define $b'(x)=\lim\limits_{i\to+\infty}b(x_i)$.

By claim 5,  $b': X\to G$ is a well-defined continuous map such that $b'(x)=b(x)$ for all $x\in X'''$.
\end{proof}

Notice that $X'''\subseteq X''\subseteq X'=M_C\subseteq D$ and $X'''$ is $\Gamma$-invariant. Combining claim 6 with the definition of $D$ appeared in the proof of claim 2, we deduce that  
$c(g, x)=b(gx)^{-1}\phi(g)b(x)=b'(gx)^{-1}\phi(g)b'(x)$ for all $g\in \Gamma$ and $x\in X'''$.

Fix any $g\in \Gamma$. Since the two continuous maps $X\ni x\mapsto c(g, x)\in G$ and $X\ni x\mapsto b'(gx)^{-1}\phi(g)b(x)\in G$ agree on $X'''$, a dense subset of $X$, they must be equal everywhere. This means that $c$ is cohomologous to $\phi$ via the continuous map $b'$.
\end{proof}

\section{Continuous orbit equivalence rigidity for generalized wreath product actions}\label{section: coe type result for generalized wreath products}

In this section, we study the continuous orbit equivalent classes of generalized wreath product actions defined below and prove Theorem \ref{thm: main thm for the whole paper}.

\begin{definition}[Generalized wreath product actions]\label{def: generalized wreath product actions}
Let $\Gamma\curvearrowright I$ be an action, where $I$ is an index set. Let $\Lambda\curvearrowright X_0$ be a continuous action. 
The generalized wreath product action $\Lambda\wr_I\Gamma\curvearrowright X:=X_0^I$, where $X_0$ is a topological space, is defined as follows:\begin{align*}
(\gamma x)_i=x_{\gamma^{-1}i},~~
[(\oplus_{i\in I}\lambda_i)x]_i=\lambda_i x_i.
\end{align*}
Here, $\gamma\in \Gamma$, $x=(x_i)_{i\in I}\in X$ and $\oplus_{i\in I}\lambda_i\in \oplus_I\Lambda=\Lambda^{(I)}$.
\end{definition}

\begin{lem}\label{lem: topological freeness for generalized wreath product actions}
Let $\Gamma\curvearrowright I$ be an action such that for every $g\in \Gamma\setminus\{e\}$, there are infinitely many $i\in I$ with $gi\neq i$. Let $\Lambda\curvearrowright X_0$ be a continuous action on a topological space $X_0$. Then $\Lambda\wr_I\Gamma\curvearrowright X:=X_0^I$ is topologically free iff $\Lambda\curvearrowright X_0$ is topologically free. If we further assume $X_0$ is a finite set (equipped with the discrete topology), then these two conditions are also equivalent to $\Lambda\curvearrowright X_0$ is free, i.e. $\lambda x_0\neq x_0$ for any $e\neq \lambda\in \Lambda$ and $x_0\in X_0$ (and hence $\Lambda$ is finite). 
\end{lem}

\begin{proof}
Let us check the equivalence in the first part.

$\Leftarrow$:
Let $ag\in \Lambda\wr_I\Gamma$ be a non-trivial element, where $g\in \Gamma$ and $a\in \oplus_I\Lambda$. We aim to show that the fixed point set $Fix(ag)$ has empty interior.

Write $supp(a)=\{i\in I: a_i\neq e\}$. Note that $x\in Fix(ag)$ iff $a_ix_{g^{-1}i}=x_i$ for all $i\in I$.

Case 1: $g\neq e$.

Assume that $Fix(ag)\supset U$, where $U$ is a standard cylinder set in $X$ determined by coordinates in a finite non-empty set $J\subset I$.
By our assumption on $\Gamma\curvearrowright I$, we may find some $i\not\in J\cup gJ\cup supp(a)$ such that $gi\neq i$. For this $i$, we have $a_i=e$ and hence $x_i=x_{g^{-1}i}$ for any $x\in Fix(ag)$. Since $i\not\in J\cup gJ$, there exists some $x\in U\subset Fix(ag)$ such that $x_i\neq x_{g^{-1}i}$, a contradiction.

Case 2: $a\neq e$ and $g=e$. 

In this case, $x\in Fix(a)$ iff $x_i=a_ix_i$ for all $i\in supp(a)$. Suppose $Fix(a)\supseteq U$ for some non-empty open set $U\subseteq X$, then we may assume $U=\{x\in X: x_j\in V_j,~\forall~j\in J\}$, where $supp(a)\subseteq J\subseteq I$ is a finite set and $\emptyset\neq V_j\subset X_0$ are open for all $j\in J$. Hence $V_i\subset Fix(a_i)$ for all $i\in supp(a)$, a contradiction.

$\Rightarrow$: Assume that $\Lambda\curvearrowright X_0$ is not topologically free, then there exists some non-trivial $\lambda\in \Lambda$ such that its fixed point set $Fix(\lambda)$ has non-empty interior. Let $V\subset Fix(\lambda)$ be some non-empty open subset. Fix any $i\in I$, let $\lambda_{(i)}$ denote the element $\lambda$ in the $i$-th component of $\oplus_I\Lambda$. Observe that $Fix(\lambda_{(i)})\supseteq \{x\in X: x_i\in V\}$, which is a non-empty open set. Hence $\Lambda\wr_I\Gamma\curvearrowright X$ is not topologically free, a contradiction.

If $X_0$ is  finite, then clearly $\Lambda\curvearrowright X_0$ is free iff it is topologically free.
\end{proof}

\begin{remark}\label{remark: minimality for generalized wreath product actions}
It is easy to check that if $\Lambda\curvearrowright X_0$ is free, then $\Lambda\wr_I\Gamma\curvearrowright X_0^I$ has no fixed points. If $\Lambda\curvearrowright X_0$ is minimal, then so is the action $\Lambda\wr_I\Gamma\curvearrowright X_0^I$ since the subaction $\oplus_I\Lambda\curvearrowright X_0^I$ is minimal.
\end{remark}

In the following lemma, we need to use the notion of infinite tensor product of C$^*$- and von Neumann algebras, see e.g.  \cite[\S 1, Chapter XIV]{Tak_book} and \cite[\S 5.1.2]{AP} for basic properties on it.
\begin{lem}\label{lem: fixed point algebra under the wm assumption}
Let $\Gamma\curvearrowright I$ be an action on an infinite set $I$ such that for every $i\in I$, the stabilizer subgroup $Stab(i)$ acts with infinite orbits on $I\setminus\{i\}$. Let $(M,\tau)$ be a  finite von Neumann algebra with a faithful normal trace $\tau$. Let $\widetilde{M}:=\bar{\otimes}_{i\in I}(M,\tau)$ be the infinite tensor product of $(M,\tau)$. Let $A\subsetneq M$ be a non-trivial C$^*$-subalgebra and $\widetilde{A}:= \otimes_{i\in I}A\subset \widetilde{M}$, the infinite (minimal) tensor product of C$^*$-algebras. Consider the generalized (non-commutative) Bernoulli shift action $\Gamma\curvearrowright \bar{\otimes}_{i\in I}M=\widetilde{M}$, i.e. $g(\otimes_ia_i)=\otimes_ia_{g^{-1}i}$ for each $g\in \Gamma$ and $\otimes_ia_i\in \widetilde{M}$ and its restriction on the C$^*$-subalgebra $\widetilde{A}$, which is $\Gamma$-invariant. 
Then the fixed point C$^*$-subalgebra of $\widetilde{A}$ under the stabilizer subgroup $Stab(i)$ is contained in $A_i$, where $A_i$ is the image of $A$ under the natural embedding into the $i$-th component in $\widetilde{A}$. 
\end{lem}
\begin{proof}
Note that $\widetilde{M}$ is also a finite von Neumann algebra and there is a unique trace on it, denoted by $\tau'$, such that $\tau'(x)=(\otimes_F\tau)(x)$ for any $x\in (\otimes_FM, \otimes_F\tau)$, where $F$ is any non-empty finite set of $I$, see \cite[\S 5.1.2]{AP}.
Therefore, we may consider the $\tau'$-preserving normal conditional expectation $E_i: \widetilde{M}\rightarrow M_i$, where $M_i$ denotes the image of $M$ under the natural embedding into the $i$-th component in $\widetilde{M}$. Then $E_i(\widetilde{A})\subset A_i$. 

To prove this, take any $x\in \widetilde{A}$. For each $\epsilon>0$, we may find a non-empty finite set $J\subset I\setminus\{i\}$, a non-empty finite subset $K$ of positive integers, some $a^{(n)}_i\in M_i$ and $b^{(n)}\in \bar{\otimes}_{j\in J}M_j$ for each $n\in K$ such that  $||a-\sum_{n\in K}a_i^{(n)}\otimes b^{(n)}||<\epsilon$. Observe that $E_i(\sum_{n\in K}a_i^{(n)}\otimes b^{(n)})=\sum_{n\in K}a_i^{(n)}\tau'(b^{(n)})\in A_i$. Since $||E_i(a)-E_i(\sum_{n\in K}a_i^{(n)}\otimes b^{(n)})||\leq ||a-\sum_{n\in K}a_i^{(n)}\otimes b^{(n)}||<\epsilon$, we deduce that $E_i(a)\in A_i$.

Take any $a\in \widetilde{A}$ such that $\gamma(a)=a$ for all $\gamma\in Stab(i)$. We aim to show $a\in A_i$. After replacing $a$ by $a-E_i(a)$, we may assume $E_i(a)=0$ and let us show $a=0$.

For any $\epsilon>0$, take $\sum_{n\in K}a_i^{(n)}\otimes b^{(n)}$ as above such that $||a-\sum_{n\in K}a_i^{(n)}\otimes b^{(n)}||<\epsilon$. 
Note that since $E_i(a)=0$, we may further assume $\tau'(b^{(n)})=0$ for each $n\in K$. Indeed,
notice that 
\begin{align*}
&||a-[\sum_{n\in K}a_i^{(n)}\otimes b^{(n)}-E_i(\sum_{n\in K}a_i^{(n)}\otimes b^{(n)})]||\\
&\leq ||a-\sum_{n\in K}a_i^{(n)}\otimes b^{(n)}||+||E_i(a)-E_i(\sum_{n\in K}a_i^{(n)}\otimes b^{(n)})||~\quad(\text{since $E_i(a)=0$})\\
&\leq 2||a-\sum_{n\in K}a_i^{(n)}\otimes b^{(n)}||\leq 2\epsilon.
\end{align*}
Moreover, $\sum_{n\in K}a_i^{(n)}\otimes b^{(n)}-E_i(\sum_{n\in K}a_i^{(n)}\otimes b^{(n)})=\sum_{n\in K}a_i^{(n)}\otimes(b^{(n)}-\tau'(b^{(n)}))$. So we may replace $b^{(n)}$ (respectively, $\epsilon$) by $b^{(n)}-\tau'(b^{(n)})$ (respectively, $2\epsilon$) to make the further assumption that $\tau'(b^{(n)})=0$ for all $n\in K$.

By our assumption on $\Gamma\curvearrowright I$, we know $Stab(i)$ acts on $I\setminus \{i\}$ with infinite orbits, so we may choose some $\gamma\in Stab(i)$ such that $\gamma J\cap J=\emptyset$ by Lemma \ref{lem: Popa-Vaes lemma}. Then,
\begin{align*}
&||a||_2^2=\langle a, a\rangle=\langle a, \gamma(a)\rangle
\overset{\epsilon(2||a||+\epsilon)}{\approx}\langle \sum_{n\in K}a_i^{(n)}\otimes b^{(n)}, \gamma(\sum_{m\in K}a_i^{(m)}\otimes b^{(m)}) \rangle\\
&=\langle \sum_{n\in K}a_i^{(n)}\otimes b^{(n)}\otimes 1, \sum_{m\in K}a_i^{(m)}\otimes 1\otimes \gamma(b^{(m)})\rangle\\
&=\sum_{n, m\in K}\tau'(a^{(n)}_i{a^{(m)}_i}^*)\cdot \tau'(b^{(n)})\cdot\tau'(\gamma({b^{(m)}}^*))=0 ~\quad\text{(since $\tau'(b^{(n)})=0$ for all $n\in K$)}.
\end{align*}
Here, the second last equality holds since $b^{(n)}\in \bar{\otimes}_{j\in J}M_j$, $\gamma(b^{(m)})\in \bar{\otimes}_{j\in \gamma J}M_j$ and $\gamma J\cap J=\emptyset$.
Since $\epsilon>0$ is arbitrary, we deduce that $a=0$.
\end{proof}
\begin{remark}\label{remark: how to apply the fixed point subalgebra}
Lemma \ref{lem: fixed point algebra under the wm assumption} will be applied in the proof of Theorem \ref{thm: main thm for generalized wreath product actions} to $\widetilde{M}=L^{\infty}(X_0^I, \mu_0^I)\rtimes (\oplus_I\Lambda)$, $\widetilde{A}=C(X_0^I)\rtimes_r(\oplus_I\Lambda)$, $M=L^{\infty}(X_0, \mu_0)\rtimes \Lambda$ and $A=C(X_0)\rtimes_r\Lambda$. Here, $\Lambda\curvearrowright X_0$ is a free action of a non-trivial finite group $\Lambda$ on a finite set $X_0$ with $|X_0|\geq 2$, $\mu_0$ is the uniform measure on $X_0$, $\mu_0^I$ is the product measure on $X_0^I$ and $\oplus_I\Lambda\curvearrowright X_0^I$ denotes the action $(\oplus_{i\in I}\lambda_i)x=(\lambda_ix_i)_{i\in I}$ for any $\oplus_{i\in I}\lambda_i\in \oplus_I\Lambda$ and  $x=(x_i)_{i\in I}\in X_0^I$. 

Note that $\widetilde{M}=U^*(\bar{\otimes}_IM)U$, where $U$ is the unitary operator defined as the composition of the following natural unitary conjugacies: 
\begin{align*}
L^2(\widetilde{M})&\cong L^2(X_0^I,\mu_0^I)\bar{\otimes}\ell^2(\oplus_I\Lambda)\\
&\cong (\otimes_I(L^2(X_0, \mu_0), 1_{X_0})\bar{\otimes}(\otimes_I(\ell^2(\Lambda), \delta_e))\\
&\cong \otimes_I(L^2(X_0,\mu_0)\bar{\otimes}\ell^2(\Lambda), 1_{X_0}\otimes\delta_e)\\
&\cong\otimes_I(L^2(M),1_{X_0}\otimes\delta_e).
\end{align*}
Moreover, for any fixed $i\in I$, the embedding $M_i\overset{Ad(U^*)}{\longrightarrow}\widetilde{M}$ is induced by the map $X_0^I\ni (x_i)_{i\in I}\mapsto x_i\in X_0$ and the inclusion $\Lambda\hookrightarrow \oplus_I\Lambda$ as the $i$-th summand.
\end{remark}
%\begin{lem}\label{lem: c*-algebra of direct sum of groupoid}
%Let $\cG_i$ be topological groupoids, then $C^*_r(\oplus_i\cG_i))\cong \otimes_iC^*_r(\cG_i)$.
%\end{lem}
%\begin{proof}
%See \cite{?}.
% It is wrong, since for groupoids, no identity exists hence the inclusion into the direct sum is not well-defined. Or more directly, direct sum of groupoids is not defined.
%\end{proof}

We fix throughout a transitive action $\Gamma\curvearrowright I$ and make the following \textbf{standing assumption}. 
\begin{enumerate}
\item[(S1)] $\Gamma\curvearrowright I$ has infinite orbits.
\item[(S2)] For every $i\in I$, the stabilizer subgroup Stab($i$) acts with infinite orbits on $I\setminus \{i\}$. For every $g\in\Gamma\setminus \{e\}$, there are infinitely many $i\in I$ with $gi\neq i$.
\item[(S3)] The group $\Gamma$ is infinite and icc.
\end{enumerate}

Note that the standing assumption is borrowed  from \cite[Theorem 5.1]{dv}. (S1) and (S2) go back to the work of Popa \cites{popa_cocycle1, popa_cocycle2} and Popa-Vaes \cite{PV_adv}, where they studied the cocycle and orbit equivalence superrigidity for generalized Bernoulli shifts. Observe that part of (S2) has appeared  in Lemma \ref{lem: topological freeness for generalized wreath product actions} and Lemma \ref{lem: fixed point algebra under the wm assumption}.

A typical example that satisfies the standing assumption is the left-right translation action, which  we record as a proposition for reference.

\begin{proposition}\label{prop: left-right translation actions satisfy the standing assumption}
Let $\Gamma=\Gamma_1\times \Gamma_1$, where $\Gamma_1$ is an infinite and icc group. Then the left-right translation $\Gamma\curvearrowright \Gamma_1:=I$ satisfies the standing assumption.
\end{proposition}
\begin{proof}
(S1) and (S3) hold since $\Gamma$ is infinite and icc. Let us check (S2) holds.

Fix any $\gamma\in I$ and $\gamma'\in I\setminus \{\gamma\}$. It is clear that $Stab(\gamma)=\{(s, \gamma^{-1}s\gamma):~s\in \Gamma_1\}$. Then the $Stab(\gamma)$-orbit of $\gamma'$ is equal to the set $\{s\gamma'\gamma^{-1} s^{-1}\gamma: ~s\in \Gamma_1\}$. Since $\gamma'\gamma^{-1}$ is non-trivial in $\Gamma_1$, the icc assumption implies the above orbit   is infinite. 

Take any non-trivial $(s, t)\in \Gamma$. Clearly, $(s, t)\gamma\neq\gamma$ iff $t\neq \gamma^{-1}s\gamma$. If $s=e$, then $t\neq e=\gamma^{-1}s\gamma$ for all $\gamma\in I$. If $s\neq e$, then the icc assumption implies $\{\gamma^{-1}s\gamma:~\gamma\in \Gamma_1\}$ is an  infinite set, and hence there are infinitely many $\gamma$ such that $t\neq \gamma^{-1}s\gamma$. Thus, (S2) holds true.  
\end{proof}

 Now, we are ready to state the main theorem in this section, which can be viewed as an analogue of a special case of \cite[Theorem 5.1]{dv} in the topological setting.

\begin{thm}\label{thm: main thm for generalized wreath product actions}
Let $X_0$ be any finite set with $|X_0|\geq 2$.
We assume the following conditions hold:
\begin{itemize}
\item $\Gamma\curvearrowright I$ is a transitive action satisfying the standing assumption. 
\item The generalized full shift $\Gamma\curvearrowright X:=X_0^I$ is a continuous cocycle superrigid action.
\item $\Lambda\curvearrowright X_0$ is a free action of a non-trivial finite group $\Lambda$.
\end{itemize} 
Let $G\curvearrowright Y$ be a topologically free continuous action. Then the following two statements are equivalent.
\begin{itemize}
\item The generalized wreath product action $\Lambda\wr_I\Gamma\curvearrowright X$ is continuously orbit equivalent to $G\curvearrowright Y$.
\item There exists some free action of a countable group $\rho: \Lambda_0\curvearrowright X_0$ such that $\rho$ and $\Lambda\curvearrowright X_0$ are continuously orbit equivalent and $G\curvearrowright Y$ is topologically conjugate to the generalized wreath product action $\Lambda_0\wr_I\Gamma\curvearrowright X_0^{I}$ defined using $\rho$ and $\Gamma\curvearrowright I$.
\end{itemize}
\end{thm}

Note that if we take $\Gamma\curvearrowright I$ to be the left-right translation action $\Gamma_1\times \Gamma_1\curvearrowright\Gamma_1$, where $\Gamma_1$ is any finitely generated, non-amenable, non-torsion and icc group, then the first two conditions in Theorem \ref{thm: main thm for generalized wreath product actions} are satisfied by Corollary \ref{main corollary} (ii) and Proposition \ref{prop: left-right translation actions satisfy the standing assumption}.

\begin{cor}\label{main corollary2}
Let $X_0$ be any finite set with $|X_0|\geq 2$. We assume the three conditions in Theorem \ref{thm: main thm for generalized wreath product actions} hold. Then there are at most finitely many topologically free continuous actions $G\curvearrowright Y$ which are continuously orbit equivalent to the generalized wreath product action $\Lambda\wr_I\Gamma\curvearrowright X_0^I$  up to topological conjugacy. In particular, if $|\Lambda|=|X_0|=p$, a prime number, then $\Lambda\wr_I\Gamma\curvearrowright X_0^I$  is a continuous orbit equivalence superrigid action.
\end{cor}
\begin{proof}
Assume that $G\curvearrowright Y$ is continuously orbit equivalent to the generalized wreath product action. Then Theorem \ref{thm: main thm for generalized wreath product actions} implies $G\curvearrowright Y$ is topologically conjugate to $\Lambda_0\wr_I\Gamma\curvearrowright X_0^I$, where $\rho:~\Lambda_0\curvearrowright X_0$ is a free action that is continuously orbit equivalent to $\Lambda\curvearrowright X_0$. From the $\Lambda$-and $\Lambda_0$-action are free, we can deduce that $|X_0|\geq |\Lambda|=|\Lambda x|=|\Lambda_0\phi(x)|=|\Lambda_0|$ for any $x\in X_0$, where $\phi$ is a homeomorphism of $X_0$ witnessing the continuous orbit equivalence. Hence, we have finitely many such groups $\Lambda_0$ up to group isomorphism. Fix any such a group $\Lambda_0$. Since the action $\rho: \Lambda_0\curvearrowright X_0$ is free, we know $\rho: \Lambda_0\rightarrow Sym(X_0)$ is injective. Clearly, we have less than $(|X_0|!)^{|\Lambda_0|}$-many such injections, and hence only finitely many such actions $\rho$ up to topological conjugacy. Obviously, this implies we have finitely many such $G\curvearrowright Y$ up to topological conjugacy. 

If $|\Lambda|=|X_0|=p$, a prime number, then any free action $\Lambda\curvearrowright X_0$ is topologically conjugate to the natural left translation$\frac{\mathbb{Z}}{p\mathbb{Z}}\curvearrowright \frac{\mathbb{Z}}{p\mathbb{Z}}$, so is the free action $\Lambda_0\curvearrowright X_0$ as $|\Lambda_0|=|\Lambda|=p$.
\end{proof}

Now, let us comment on the proof of Theorem \ref{thm: main thm for generalized wreath product actions}, which is based on the proof of \cite[Theorem 5.1]{dv} but with the following modifications/differences:

1. In \cite{dv}, the proof was written using equivalence relations in the measurable setting, but we prefer using group actions and sometimes topological groupoids whenever convenience.

2. We are not able to translate the proof in \cite{dv} (even with trivial amplification, i.e. $t=1$ there) directly to a proof in the topological setting. For example, it used Popa's cocycle superrigidity theorem for non-commutative generalized Bernoulli shifts, which is not available in the topological setting.

3. To circumvent the above issue, the key observation is that a natural modification of the proof of steps 1-4 in \cite{dv} actually shows we may replace the initial coe coupling with a new one, which simplifies the proof of following steps, see the sentence before step 7 in our proof. This is the main difference between our proof with the one in \cite{dv} in our opinion.

\begin{proof}[Proof of Theorem \ref{thm: main thm for generalized wreath product actions}]
It is easy to see the second statement implies the first. Indeed, suppose $(c, c', \phi)$ is a coe coupling for the continuous orbit equivalence between $\rho: \Lambda_0\curvearrowright X_0$ and $\Lambda\curvearrowright X_0$. Then $(\widetilde{c}, \widetilde{c'}, \widetilde{\phi})$ is a coe coupling for the generalized wreath product action $\Lambda_0\wr_I\Gamma\curvearrowright X_0^I$ and $\Lambda\wr_I\Gamma\curvearrowright X_0^I$. Here, $\widetilde{\phi}$ is the homeomorphism of $X_0^I$ defined by $\widetilde{\phi}((x_i)_{i\in I})=(\phi(x_i))_{i\in I}$, where $(x_i)_{i\in I}\in X_0^I$. $\widetilde{c}: \Lambda_0\wr_I\Gamma\times X_0^I\to \Lambda\wr_I\Gamma$ is the cocycle defined by $\widetilde{c}(\gamma, (x_i)_{i\in I})=\gamma$ and $\widetilde{c}(\oplus_{i\in I}\lambda_i, (x_i)_{i\in I})=\oplus_{i\in I}c(\lambda_i, x_i)$, where $\gamma\in \Gamma$, $\oplus_{i\in I}\lambda_i\in \oplus_I\Lambda_0$ and $(x_i)_{i\in I}\in X_0^I$. $\widetilde{c'}$ is defined similarly using $c'$.

So we assume that the first statement holds.

Let $(c, c', \varphi)$ be a coe coupling between the two initial actions. More precisely, $\varphi: X\to Y$ is a homeomorphism witnessing the continuous orbit equivalence between $\Lambda\wr_I\Gamma\curvearrowright X$ and $G\curvearrowright Y$, $c: \Lambda\wr_I\Gamma\times X\to G$ is the orbit cocycle and $c': G\times Y\to \Lambda\wr_I\Gamma$ is the inverse orbit cocycle, see the paragraph right after Definition \ref{def: coe} for the meaning of inverse cocycles.

Let $\mu=\mu_0^I$ be the product measure on $X$, where $\mu_0$ is the uniform measure on the finite set $X_0$, i.e. $\mu(\{a\})=\frac{1}{|X_0|}$ for every $a\in X_0$. Clearly, $\mu$ is $\Lambda\wr_I\Gamma$-invariant. Observe that $\mu(U)>0$ for any non-empty open set in $X$. In the proof, we write $\Lambda^{(I)}$ for $\oplus_I\Lambda$ sometimes without mention.

We number the steps in the proof for ease of reference.

\textbf{Step 1}. By our assumption,
there exists some continuous map $L: X\to G$ and a group homomorphism $\delta: \Gamma\to G$ such that $c(g, x)=L(gx)^{-1}\delta(g)L(x)$ for all $g\in \Gamma$ and $x\in X$. Let $\phi(x)=L(x)\varphi(x)$. Note that $\phi(gx)=\delta(g)\phi(x)$ for all $g\in \Gamma$ and $x\in X$.
The key difficulty is to show $\phi$ is still a homeomorphism. It is clear that $\phi$ is continuous and the main issue is to show it is a bijection, which will be done in step 4.

For any $g\in \Lambda\wr_I\Gamma$ and $x\in X$, define
\begin{align}\label{def: definition of the cocycle omega}
\omega(g, x)=L(gx)c(g, x)L(x)^{-1}.
\end{align}
Clearly, 
\begin{align}\label{eq: omega(g,x)=delta(g)}
\omega(g, x)=\delta(g)~\text{for all $g\in \Gamma$ and}~\phi(gx)=\omega(g, x)\phi(x)~\text{for all $g\in \Lambda\wr_I\Gamma$.}
\end{align}

Let us show that $Ker(\delta)=\{e\}$. 

Take any $g\in Im(L)<G$ and let $A=L^{-1}(\{g\})\subset X$. We observe that $A\cap s^{-1}A=\emptyset$ for any $e\neq s\in Ker(\delta)$. Assume not, then take any $x\in A\cap s^{-1}A$, i.e. $L(x)=g=L(sx)$ and we deduce that $c(s, x)=L(sx)^{-1}\delta(s) L(x)=e$. Since for each $x\in X$, $\Lambda\wr_I\Gamma\ni g\mapsto c(g, x)\in G$ is a bijection, we get $s=e$, a contradiction. Note that $\mu(A)>0$ as $A$ is a non-empty open set. Since $\mu$ is $\Gamma$-invariant, we deduce that $Ker(\delta)$ is finite. Since $\Gamma$ is assumed to be icc by (S3), we get that $Ker(\delta)=\{e\}$. 

\textbf{Step 2}. For any $\lambda\in \Lambda$ and $i\in I$, we write $\lambda_{(i)}$ to mean the element $\lambda$ in the $i$-th summand of $\oplus_I\Lambda<\Lambda\wr_I\Gamma$. We prove that
\begin{align}\label{eq: omega descends to omega_i}
\omega(\lambda_{(i)}, x)=\omega_i(\lambda, x_i)~\text{for all $x\in X$ and all $\lambda\in\Lambda$},
\end{align}
where $\omega_i: \Lambda\times X_0\to G$ is a cocycle.

For each $\lambda\in \Lambda$, define $F: X\to G$ by $F(x)=\omega(\lambda_{(i)}, x)$. It suffices to show $F$ only depends on the $i$-th coordinate of $x$.

Fix any $g\in Stab(i)<\Gamma$. We compute
\begin{align*}
F(gx)&=\omega(\lambda_{(i)}, gx)\\
&=\omega(\lambda_{(i)}g, x)\omega(g, x)^{-1}~\quad\text{(by cocycle relation)}\\
&=\omega(g\lambda_{(i)}, x)\omega(g, x)^{-1}~\quad\text{(since $g\in Stab(i)$ implies $g\lambda_{(i)}=\lambda_{(i)}g$)}\\
&=\omega(g, \lambda_{(i)}x)\omega(\lambda_{(i)}, x)\omega(g, x)^{-1}~\quad\text{(by cocycle relation)}\\
&=\delta(g)F(x)\delta(g)^{-1}~\quad\text{(by \eqref{eq: omega(g,x)=delta(g)} and def. of $F$)}.
\end{align*}
Our goal is to show that for any $x, y\in X$, if $x_i=y_i$, then $F(x)=F(y)$. Note that we may think of $F$ as defined on $X_0^{I\setminus \{i\}}$. Indeed, by fixing an arbitrary $x_0\in X_0$, we may replace $F$ by $F'(x):=F(\overline{x_0 x})$ for all $x\in X_0^{I\setminus \{i\}}$, which also satisfies that $F'(gx)=\delta(g)F'(x)\delta(g)^{-1}$. Here, $\overline{x_0x}$ denotes the natural element in $X_0^I=X_0\times X_0^{I\setminus\{i\}}$, i.e. $(\overline{x_0x})_j=x_0$ if $j=i$ and $x_j$ otherwise.

In other words, we need to show $F$ is constant on $X_0^{I\setminus \{i\}}$.

Let $\theta(x, y):=F(x)F(y)^{-1}$ for all $x, y\in X_0^{I\setminus \{i\}}$. 
Since $Stab(i)$ acts with infinite orbits on $I\setminus \{i\}$ by (S2), the action of $Stab(i)$ on $X_0^{I\setminus \{i\}}$ is topologically weakly mixing by Lemma \ref{lem: top.w.m. for generalized full shifts}, i.e. the diagonal action $Stab(i)\curvearrowright X_0^{I\setminus \{i\}}\times X_0^{I\setminus\{i\}}$ is (topologically) transitive. 
Note that $\theta(gx, gy)=\delta(g)\theta(x, y)\delta(g)^{-1}$, hence $\theta(gx, gy)$ and $\theta(x, y)$ lie in the same conjugacy class in $G$ for all $g\in Stab(i)$. Notice that $F$ is continuous, so is $\theta$. Thus the diagonal action is topologically transitive implies that $\theta(x, y)$ and $\theta(x, x)$ lie in the same conjugacy class in $G$ for all $x, y$. Notice that $\theta(x, x)=e$, hence $\theta(x, y)=e$, i.e. $F(x)=F(y)$.

\textbf{Step 3}. Let $\Lambda_i=\langle \omega(\lambda_{(i)}, x):~x\in X,~\lambda\in \Lambda\rangle<G$. Let us check that $[\Lambda_i,\Lambda_j]=\{e\}$ for all $i\neq j$, i.e. elements in $\Lambda_i$ commute with those in $\Lambda_j$.

Take any $\lambda, \lambda'\in \Lambda$ and $x, y\in X$. Take any $z\in X$ such that $z_j=y_j$, $z_i=x_i$ and set $w:=\lambda'_{(j)}z\in X$. Note that $w_i=x_i$.
\begin{align*}
\omega(\lambda_{(i)}, x)\omega(\lambda'_{(j)}, y)
&=\omega_i(\lambda, x_i)\omega_j(\lambda', y_j)~\quad\text{(by \eqref{eq: omega descends to omega_i})}\\
&=\omega_i(\lambda, w_i)\omega_j(\lambda', z_j)~\quad\text{(since $x_i=z_i=\omega_i$ and $y_j=z_j$)}\\
&=\omega(\lambda_{(i)}, w)\omega(\lambda'_{(j)}, z)~\quad\text{(by \eqref{eq: omega descends to omega_i})}\\
&=\omega(\lambda_{(i)}\lambda'_{(j)}, z)~\quad\text{(by def. of $\omega$ and cocycle relation)}\\
&=\omega(\lambda'_{(j)}\lambda_{(i)}, z)~\quad\text{(since $i\neq j$ implies $\lambda_{(i)}\lambda'_{(j)}=\lambda'_{(j)}\lambda_{(i)}$)}\\
&=\omega(\lambda'_{(j)}, \lambda_{(i)}z)\omega(\lambda_{(i)}, z)~\quad\text{(by cocycle relation)}\\
&=\omega_j(\lambda', [\lambda_{(i)}z]_j)\omega_i(\lambda, z_i)~\quad\text{(by \eqref{eq: omega descends to omega_i})}\\
&=\omega_j(\lambda', y_j)\omega_i(\lambda, x_i)~\quad\text{(since $[\lambda_{(i)}z]_j=z_j=y_j, z_i=x_i$)}\\
&=\omega(\lambda'_{(j)}, y)\omega(\lambda_{(i)}, x)~\quad\text{(by \eqref{eq: omega descends to omega_i})}.
\end{align*}
Therefore, $[\Lambda_i, \Lambda_j]=\{e\}$.

Note that $\Lambda_{gi}=\delta(g)\Lambda_i\delta(g)^{-1}$ for all $g\in \Gamma$. Indeed, for each $\lambda\in \Lambda$, we have 
\begin{align*}
\omega(\lambda_{(gi)}, x)=\omega(g\lambda_{(i)}g^{-1}, x)
&=\omega(g, \lambda_{(i)}g^{-1}x)\omega(\lambda_{(i)}, g^{-1}x)\omega(g^{-1}, x)\\
&=\delta(g)\omega(\lambda_{(i)}, g^{-1}x)\delta(g)^{-1}~\text{by \eqref{eq: omega(g,x)=delta(g)}}. 
\end{align*}
The above calculation and \eqref{eq: omega descends to omega_i} imply that $[\Lambda_i, \delta(Stab(i))]=\{e\}$. 

Fix $i_0\in I$ and write $\Lambda_0=\Lambda_{i_0}$. Define $G_2=\Lambda_0\wr_I\Gamma$ and denote, for each $i\in I$, by $\pi_i: \Lambda_0\to \Lambda_0^{(I)}<G_2$ the embedding as the $i$-th direct summand. Since $\Gamma\curvearrowright I$ is transitive and $\Lambda_i$ commutes with both $\Lambda_j$ and $\delta(Stab(i))$, we deduce that the group homomorphism $\delta: \Gamma\to G$ can be uniquely extended to a group homomorphism $\delta: G_2\to G$ satisfying $\delta(\pi_{i_0}(\lambda))=\lambda$ for all $\lambda\in \Lambda_0$. We also find a unique 1-cocycle $\omega_2: \Lambda\wr_I\Gamma\times X\to G_2$ satisfying that 
\begin{align}\label{def: def of omega_2}
\omega_2(g, x)=g~\mbox{if $g\in \Gamma$, and}~\omega_2(\lambda_{(i_0)}, x)=\pi_{i_0}(\omega_{i_0}(\lambda, x_{i_0}))~\mbox{for all $\lambda\in \Lambda$}.
\end{align}

By construction, $\omega=\delta\circ\omega_2$. Also note that by construction,
\begin{align}\label{property: property of omega_2}
\begin{split}
\omega_2(g'a', x)=ga~ \mbox{with $g, g'\in \Gamma$ , $a'\in \Lambda^{(I)}$ and $a\in \Lambda_0^{(I)}$}\\
\Leftrightarrow~
g'=g~\mbox{and}~a_i=\omega_{i_0}(a'_i, x_i)~\mbox{for all $i\in I$}.
\end{split}
\end{align}
Here, we have written $a'=\oplus_{i\in I}a'_i\in \Lambda^{(I)}=\oplus_I\Lambda$, $a=\oplus_{i\in I}a_i\in \Lambda_0^{(I)}=\oplus_I\Lambda_0$, where $a_i'\in \Lambda$ and $a_i\in \Lambda_0$ for all $i\in I$. 

Below, we check that $\omega_2(\lambda_{(i)}, x)=\pi_i(\omega_{i_0}(\lambda, x_i))$ for every $i$. Once we know this, then \eqref{property: property of omega_2} follows by routine calculations using cocycle relation and \eqref{def: def of omega_2}.  

Take any $g\in \Gamma$ such that $gi_0=i$. Then
\begin{align*}
\omega_2(\lambda_{(i)}, x)
&=\omega_2(\lambda_{(gi_0)}, x)\\
&=\omega_2(g\lambda_{(i_0)}g^{-1}, x)\\
&=\omega_2(g, \lambda_{(i_0)}g^{-1}x)\omega_2(\lambda_{(i_0)}, g^{-1}x)\omega_2(g^{-1}, x)\\
&=g\pi_{i_0}(\omega_{i_0}(\lambda, (g^{-1}x)_{i_0}))g^{-1}~\quad\text{(by \eqref{def: def of omega_2})}\\
&=g\pi_{i_0}(\omega_{i_0}(\lambda, x_i))g^{-1}~\quad\text{(since $(g^{-1}x)_{i_0}=x_{gi_0}=x_i$)}\\
&=\pi_i(\omega_{i_0}(\lambda, x_i))~\quad\text{(since $g\pi_{i_0}(\cdot)g^{-1}=\pi_{gi_0}(\cdot)=\pi_i(\cdot)$)}.
\end{align*}

\textbf{Step 4}. Let us show $\phi: X\to Y$ is a bijection.

(I) We first show $\phi$ is injective. 

Assume that $\phi(x)=\phi(x')$ for two points $x, x'\in X$. By the definition of $\phi$, this means $L(x)\varphi(x)=L(x')\varphi(x')$, i.e. $L(x')^{-1}L(x)\varphi(x)=\varphi(x')$. Since $\Lambda\wr_I \Gamma \ni g\mapsto c(g, x)\in G$ is a bijection, we may take $s\in \Lambda\wr_I\Gamma$ such that $c(s, x)=L(x')^{-1}L(x)$, thus $\varphi(x')=\varphi(sx)$, i.e. $x'=sx$. Hence, $c(s, x)=L(sx)^{-1}L(x)$; equivalently, $\omega(s, x)=e$. We are left to show $s=e$. The proof given in step 4 in the proof of \cite[Theorem 5.1]{dv} still works here. We include it here for completeness.

Below, we prove that if $\omega(g, x)=e$ for some $g\in \Lambda\wr_I\Gamma$ and some $x\in X$, then $g=e$.

First, observe that $n_x:=\sharp\{g\in \Lambda\wr_I\Gamma:~\omega(g, x)=e\}<\infty$. To see this, note that $\omega(g, x)=e$ iff $c(g, x)=L(gx)^{-1}L(x)$, which takes finitely many values since $L$ is a continuous map from the compact space $X$ to the discrete group $G$. Then we may apply the fact that $g\mapsto c(g, x)$ is injective to conclude $n_x<\infty$.

Clearly, $1\leq n_x$, and observe that $n_{\gamma x}=n_x$ for all $\gamma\in \Gamma$. Indeed, one can check $\gamma \{g\in \Lambda\wr_I\Gamma:~\omega(g, x)=e\}\gamma^{-1}=\{g\in \Lambda\wr_I\Gamma:~\omega(g, \gamma x)=e\}$.

Since $\Gamma\curvearrowright (X, \mu)$ is ergodic, we deduce that $n_x$ is a constant, say $n$ for a.e. $x\in X$. 

Define $\xi: X\to \ell^2(\Lambda\wr_I\Gamma)$ by setting $\xi(x)=\sum_{g:~\omega(g, x)=e}\delta_g$ for a.e. $x\in X$. Note that $||\xi||_2^2=\int_X||\xi(x)||_2^2d\mu(x)=\int_X nd\mu(x)=n<\infty$, thus $\xi\in L^2(X,\ell^2(\Lambda\wr_I\Gamma))\cong \cH:=L^2(X, \mu)\overline{\otimes} \ell^2(\Lambda\wr_I\Gamma)$. Consider the unitary representation $\pi: \Gamma\curvearrowright \cH$ by $\pi(\gamma)(f\otimes \delta_s)=\gamma \cdot f\otimes\delta_{\gamma s\gamma^{-1}}$ for any $s\in \Lambda\wr_I\Gamma$, $f\in L^2(X,\mu)$ and $\gamma\in \Gamma$. Then one can check that $\xi$ is $\pi(\Gamma)$-invariant. 

Moreover, observe that $\pi$ is unitarily conjugate to the tensor product of the representations 
\begin{align*}
& Ad:~\Gamma\curvearrowright \ell^2(\Gamma):~(Ad_g\xi)(h)=\xi(g^{-1}hg)~\mbox{and}~\\
&\pi_1:~\Gamma\curvearrowright \overline{\otimes}_{i\in I}(L^2(X_0,\mu_0)\overline{\otimes} \ell^2(\Gamma), 1_{X_0}\otimes \delta_e),
\end{align*}
where $\pi_1$ acts by permuting the tensor factors of the infinite tensor product of the Hilbert space $L^2(X_0,\mu_0)\otimes \ell^2(\Gamma)$ w.r.t. the canonical unit vector $1_{X_0}\otimes \delta_e$. Since $\Gamma$ is icc by (S3) and $\Gamma\curvearrowright I$ has infinite orbits by (S1), the only $\Gamma$-invariant vectors for $Ad\otimes \pi_1$ are the multiples of $\delta_e\otimes 1$. Thus, the only $\pi(\Gamma)$-invariant vectors in $\cH$ are the multiples of  $(x\mapsto \delta_e)\in L^2(X, \ell^2(\Lambda\wr_I\Gamma))\cong \cH$. This means that $n=1$. 

On the other hand, observe that if $x_i\to x$, then $n_x\leq \limsup n_{x_i}$. Since $\{x:~n_x=1\}$ has $\mu$-measure 1, it is also dense. Then for every $x\in X$, $1\leq n_x\leq 1$, i.e. $n_x=1$ for all $x\in X$. Thus $\omega(g, x)=e$ implies $g=e$.

This finishes showing that $\phi$ is injective. In fact, observe that we have also proved the map $\omega(-, x):~\Lambda\wr_I\Gamma\rightarrow G$ is injective for every $x\in X$.

(II) We show that $\phi: X\to Y$ is onto.

Write $\nu=\varphi_*\mu$ and $X=\sqcup_{g\in G}X_g$, where $X_g=\{x\in X: L(x)=g\}$. Then
$\phi(X)=\sqcup_{g\in G}g\varphi(X_g)$ since $\phi$ is injective as proved in (I). Let us check that $\nu(\phi(X))=1$.

Indeed, write $Y=\sqcup_{t\in \Lambda\wr_I\Gamma}Y_{t,g}$, where $Y_{t, g}=\{y\in Y:~c'(g, y)=t\}$. Then for each $g\in G$, we have
\begin{align*}
&\nu(g\varphi(X_g))\\
&=\mu(\varphi^{-1}(g\varphi(X_g)))~\quad\text{(since $\nu=\phi_*\mu$)}\\
&=\mu(\varphi^{-1}(g\sqcup_{t\in \Lambda\wr_I\Gamma}(\varphi(X_g)\cap Y_{t, g})))~\quad\text{(since $\varphi(X_g)\subseteq Y=\sqcup_tY_{t, g}$)}\\
&=\mu(\sqcup_{t\in \Lambda\wr_I\Gamma}\varphi^{-1}(g(\varphi(X_g)\cap Y_{t, g})))~\quad\text{(since $\varphi$ is a homeomorphism)}\\
&=\mu(\sqcup_{t\in \Lambda\wr_I\Gamma}t\varphi^{-1}(\varphi(X_g)\cap Y_{t, g}))~\quad\text{(use $(c,c',\varphi)$ is a coe coupling and def. of $Y_{t, g}$)}\\
&=\sum_{t\in \Lambda\wr_I\Gamma}\mu(t\varphi^{-1}(\varphi(X_g)\cap Y_{t, g}))\\
&=\sum_{t\in \Lambda\wr_I\Gamma}\mu(\varphi^{-1}(\varphi(X_g)\cap Y_{t, g}))\\
&=\mu(\sqcup_{t\in \Lambda\wr_I\Gamma} \varphi^{-1}(\varphi(X_g)\cap Y_{t, g}))\\
&=\mu(X_g\cap \varphi^{-1}(\sqcup_{t\in \Lambda\wr_I\Gamma}Y_{t, g}))=\mu(X_g\cap \varphi^{-1}(Y))=\mu(X_g).
\end{align*}
Thus, $\nu(g\varphi(X_g))=\mu(X_g)$. Using it, we can continue the computation as follows: 
\begin{align*}
\nu(\phi(X))=\nu(\sqcup_{g\in G}g\varphi(X_g))=\sum_{g\in G}\nu(g\varphi(X_g))=\sum_{g\in G}\mu(X_g)=\mu(\sqcup_{g\in G}X_g)=\mu(X)=1.
\end{align*}
Thus, $\mu(\varphi^{-1}(Y\setminus \phi(X))=0$, which implies $\varphi^{-1}(Y\setminus \phi(X))$ is empty as it is an open set. Hence, $\phi$ is surjective.

\textbf{Step 5}. In step 3, we defined the group homomorphism $\delta: G_2\to G$. Put $\Sigma_2=Ker(\delta)$. We prove that $\Sigma_2\subset \Lambda_0^{(I)}$. 

Assume that $ga\in \Sigma_2$ with $g\in \Gamma\setminus\{e\}$ and $a\in \Lambda_0^{(I)}$. Take a finite subset $J\subset I$ such that $a\in \Lambda_0^{J}$. Since there are infinitely many $i\in I$ with $gi\neq i$ by (S2), we can choose $i\in I\setminus J$ such that $gi\neq i$. Put $j=gi$. Since $\Sigma_2$ is a normal subgroup of $G_2$, for every $\lambda\in \Lambda_0$, we have $\pi_i(\lambda)ga\pi_i(\lambda)^{-1}\in \Sigma_2$. Since $i\not\in J$, we have 
$\pi_i(\lambda)ga\pi_i(\lambda)^{-1}=\pi_i(\lambda)\pi_j(\lambda^{-1})ga\in \Sigma_2$,
so that $\pi_i(\lambda)\pi_j(\lambda^{-1})\in \Sigma_2$ for all $\lambda\in \Lambda_0$.

Now, since $\Lambda\curvearrowright X_0$ is non-trivial, we may take $e\neq t\in\Lambda$ and $x_0\in X_0$ such that $t x_0\neq x_0$. Write $\lambda=\omega_{i_0}(t, x_0)\in \Lambda_0$.

Let $x\in X$ be any point such that $x_i=x_0$ and $x_j=tx_0$. Let $\hat{a}\in \Lambda^{(I)}<\Lambda\wr_I\Gamma$ such that $\hat{a}_i=t$, $\hat{a}_j=t^{-1}$ and $\hat{a}_k=e$ for all $k\neq i, j$. Then by the definition of $\omega_2$, we have $\omega_2(\hat{a}, x)=\pi_i(\lambda)\pi_j(\lambda)^{-1}$. Hence $\omega(\hat{a}, x)=\delta(\omega_2(\hat{a}, x))=e$. As $\omega(-, x)$ is injective by step 4 (I), we deduce that $\hat{a}=e$, a contradiction.

\textbf{Step 6}. We prove that $\delta: G_2\to G$ is surjective.

Since $\delta\circ \omega_2=\omega$, it suffices to show $\{\omega(g, x): ~g\in \Lambda\wr_I\Gamma, x\in X\}=G$.

Fix any $g\in G$ and $x\in X$, set $x'=\phi^{-1}(g\phi(x))\in X$ (recall that $\phi: X\to Y$ is a bijection as proved in step 4). Then 
$\phi(x')=g\phi(x)$, i.e. $\varphi(x')=L(x')^{-1}gL(x)\varphi(x)$. Since $\Lambda\wr_I\Gamma\ni g\mapsto c(g, x)\in G$ is a bijection, we may find a unique $s\in \Lambda\wr_I\Gamma$ such that $c(s, x)=L(x')^{-1}gL(x)$. Then $\varphi(x')=c(s, x)\varphi(x)=\varphi(sx)$, i.e. $x'=sx$. Hence $\omega(s, x)=L(sx)c(s, x)L(x)^{-1}=g$ by \eqref{def: definition of the cocycle omega}. This shows that in fact the map $\omega(-, x)$ is onto for any $x\in X$. To sum up, we have shown that
\begin{align}\label{fact: the new omega(-, x) is a bijection}
\text{The map}~\Lambda\wr_I\Gamma\ni s\mapsto \omega(s, x)\in G~\text{is a bijection for any $x\in X$}. 
\end{align}

Define $\omega': G\times Y\to \Lambda\wr_I\Gamma$ by setting $\omega'(g, y)=s$ where $s$ is the element we get in the above procedure by starting with $x:=\phi^{-1}(y)$. Clearly,  $\omega(s, \phi^{-1}(y))=g$.

It is routine to check that $\omega'$ is a well-defined continuous cocycle and we also have $\phi^{-1}(gy)=\omega'(g, y)\phi^{-1}(y)$ for any $g\in G$ and $y\in Y$.
From this identity, the fact that $\phi$ is a homeomorphism by step 4 and together with \eqref{eq: omega(g,x)=delta(g)}, we deduce that $(\omega,\omega',\phi)$ is a new coe coupling between the two initial actions.
For the following steps, this new coe coupling will be needed.

\textbf{Step 7}. We prove that $\delta: G_2\to G$ is also injective. 

Recall that $(\omega, \omega',\phi)$ is a coe coupling for the continuous orbit equivalence between $\Lambda\wr_I\Gamma\overset{\star}{\curvearrowright} X_0^I=X$ and $G\overset{*}{\curvearrowright} Y$. To avoid confusion in this step, we have denoted the two original actions by $\star$ and $*$ respectively.
By Theorem \ref{thm: coe by isomorphic topological groupoids}, we know $(\Lambda\wr_I\Gamma)\ltimes X\cong G\ltimes Y$ as topological groupoids via the map 
\begin{align}\label{def: definition of Phi for isomorphism of groupoids}
\Phi((a'g, x))=(\omega(a'g, x), \phi(x)),~ \text{where $g\in \Gamma$, $a'\in \Lambda^{(I)}$ and $x\in X$}.
\end{align}

We claim that $\Phi(\Lambda^{(I)}\ltimes X)= \delta(\Lambda_0^{(I)})\ltimes Y$, where both sides are considered as subgroupoids.

\begin{proof}
For any $g\in\Gamma$, $a'\in \Lambda^{(I)}$ and $x\in X$, we have
\begin{align}\label{eq: restriction of iso on subgroupoids}
\begin{split}
\Phi((a'g, x))&=(\omega(a'g, x), \phi(x))~\quad\text{(by \eqref{def: definition of Phi for isomorphism of groupoids})}\\
&=(\omega(a', gx)\omega(g, x),\phi(x))~\quad\text{(by cocycle relation)}\\
&=(\delta(\omega_2(a', gx))\delta(g), \phi(x))~\quad\text{(by \eqref{eq: omega(g,x)=delta(g)} and $\delta\circ\omega_2=\omega$)}.
\end{split}
\end{align}
Set $g=e$ in \eqref{eq: restriction of iso on subgroupoids}, we get 
$\Phi((a', x))=(\delta(\omega_2(a', x)), \phi(x))\in \delta(\Lambda_0^{(I)})\ltimes Y$ from \eqref{property: property of omega_2}. This shows $\subseteq$ in the claim holds.

For the reverse inclusion, we take any $a\in \Lambda_0^{(I)}$ and $x\in X$. Since $\Phi$ is an isomorphism, we may find $a'g\in \Lambda\wr_I\Gamma$ such that $\Phi((a'g, x))=(\delta(a), \phi(x))$, where $a'\in \Lambda^{(I)}$ and $g\in \Gamma$.
From \eqref{eq: restriction of iso on subgroupoids}, this means we have $\delta(\omega_2(a', gx))\delta(g)=\delta(a)$, i.e. 
\[g\in \omega_2(a', gx)^{-1}a\cdot Ker(\delta)\subseteq \Lambda_0^{(I)}\] by \eqref{property: property of omega_2} and $Ker(\delta)\subset \Lambda_0^{(I)}$ as proved in step 5.

Therefore, $g=e$. Hence, $(\delta(a), \phi(x))=\Phi(a', x)\in \Phi(\Lambda^{(I)}\ltimes X)$.
\end{proof}
Define $\beta: \Gamma\curvearrowright G\ltimes Y$ by 
\begin{align}\label{def: definition of beta-action}
\beta_{\gamma}(g, y)=(\delta(\gamma)g\delta(\gamma^{-1}), \delta(\gamma)*y)~\text{for each $\gamma\in\Gamma$ and $(g, y)\in G\ltimes Y$}.
\end{align}
 Observe that the subgroupoid $\delta(\Lambda_0^{(I)})\ltimes Y$ is invariant under this action. Define an action $S: \Gamma\curvearrowright\Lambda^{(I)}\ltimes X$ by 
\begin{align}\label{def: definition of shift-action S}
S_{\gamma}(a', x)=(\gamma a'\gamma^{-1}, \gamma\star x), ~\text{where $a'\in \Lambda^{(I)}$ and $x\in X$}. 
\end{align} 
Then
\begin{align}\label{relation: Phi intertwins S and beta}
\Phi\circ S_{\gamma}=\beta_{\gamma}\circ \Phi~\text{on $\Lambda^{(I)}\ltimes X$ for each $\gamma\in \Gamma$}.
\end{align}
Indeed, this is based on the following computation. Let $(a',x)\in \Lambda^{(I)}\times X$. Then 
\begin{align*}
\Phi(S_{\gamma}(a', x))&\overset{\eqref{def: definition of shift-action S}}{=}\Phi((\gamma a'\gamma^{-1}, \gamma\star x))\\
&\overset{\eqref{def: definition of Phi for isomorphism of groupoids}}{=}(\omega(\gamma a'\gamma^{-1}, \gamma\star x), \phi(\gamma\star x))\\
&=(\omega(\gamma, a'\star x)\omega(a', x)\omega(\gamma^{-1}, \gamma\star x), \omega(\gamma, x)*\phi(x))\\
&\overset{\eqref{eq: omega(g,x)=delta(g)}}{=}(\delta(\gamma)\omega(a', x)\delta(\gamma)^{-1}, \delta(\gamma)*\phi(x))\\
&\overset{\eqref{def: definition of beta-action}}{=}\beta_{\gamma}(\omega(a', x), \phi(x))\overset{\eqref{def: definition of Phi for isomorphism of groupoids}}{=}\beta_{\gamma}(\Phi(a',x)).
\end{align*}
Therefore, $\Phi$ induces a $*$-isomorphism between the reduced groupoid $C^*$-algebras  
\[\Phi^*: C^*_r(\delta(\Lambda_0^{(I)})\ltimes Y)\cong C^*_r(\Lambda^{(I)}\ltimes X).\] 
Clearly, $\beta$ and $S$ also extend to $\Gamma$-actions on these C$^*$-algebras, which we denoted by $\beta^*$ and $S^*$ respectively.

Next, observe that $C_r^*(\Lambda^{(I)}\ltimes X)\cong \otimes_IC^*_r(\Lambda\ltimes X_0)$ (here we use the minimal tensor product).

To see this, first note that the natural projection onto the $i$-th coordinate $X_0^I\ni (x_j)_{j\in I}\mapsto x_i\in X_0$ and the inclusion into the $i$-th component $\Lambda\hookrightarrow \Lambda^{(I)}$ induce a natural inclusion, denoted by $\phi_i$, of C$^*$-algebras $C(X_0)\rtimes_r\Lambda\hookrightarrow C(X_0^I)\rtimes_r\Lambda^{(I)}$. Clearly, $\phi_i(C(X_0)\rtimes_r\Lambda)$ commutes with $\phi_j(C(X_0)\rtimes_r\Lambda)$ for all $i\neq j$ and $\{\phi_i(C(X_0)\rtimes_r\Lambda): i\in I\}$ generates $C(X_0^I)\rtimes_r\Lambda^{(I)}$. Thus we get the following isomorphism.
\begin{align*}
C^*_r(\Lambda^{(I)}\ltimes X)=C^*_r(\Lambda^{(I)}\ltimes X_0^I)&\cong C(X_0^I)\rtimes_r\Lambda^{(I)}\\
&\cong \otimes_I[C(X_0)\rtimes_r\Lambda]~ \quad\mbox{(by \cite[Proposition 11.4.3]{KR_II})}\\
&\cong \otimes_IC^*_r(\Lambda\ltimes X_0).
\end{align*}
Moreover, it is routine to check that the above isomorphism intertwines the induced action $S^*:~\Gamma\curvearrowright C^*_r(\Lambda^{(I)}\ltimes X)$ with the generalized (non-commutative) Bernoulli shift action $\Gamma\curvearrowright\otimes_IC^*_r(\Lambda\ltimes X_0)$ via $\Gamma\curvearrowright I$.

Fix any $i\in I$. Denote by $\Pi_i: C_r^*(\Lambda\ltimes X_0)\hookrightarrow C^*_r(\Lambda^{(I)}\ltimes X)\cong \otimes_IC^*_r(\Lambda\ltimes X_0)$ the embedding as the $i$-th tensor factor. 

The following diagram shows the natural isomorphisms and inclusions we have between the above C$^*$-algebras. Note that the action on the bottom line is the generalized Bernoulli shift via $\Gamma\curvearrowright I$ and the two isomorphisms $\cong$ intertwine the corresponding actions.
		\begin{align}\label{figure1}
\xymatrix{
\Gamma\overset{\beta^*}{\curvearrowright} C^*_r(\delta(\Lambda_0^{(I)})\ltimes Y)\ar[d]^{\Phi^*}_{\cong}&\\
\Gamma\overset{S^*}{\curvearrowright}C^*_r(\Lambda^{(I)}\ltimes X) \ar[d]_{\cong}&C^*_r(\Lambda\ltimes X_0)\ar[l]\ar@{^{(}->}[dl]^{\Pi_i}\\
\Gamma\curvearrowright\otimes_IC^*_r(\Lambda\ltimes X_0)&}
\end{align}

Let us check that $\Phi^*(u_{\delta(\pi_i(\lambda))})$ is invariant under the action $S^*|_{Stab(i)}$.   

Recall that $\pi_i: \Lambda_0\hookrightarrow \Lambda_0^{(I)}$ denotes the embedding into the $i$-th component. It is clear that
$\gamma$ commutes with $\pi_i(\lambda)$ for all $\lambda\in \Lambda_0$ and all $\gamma\in Stab(i)$.
Fix any $\gamma\in Stab(i)$ and $(g, y)\in \delta(\Lambda_0^{(I)})\ltimes Y$. Recall that $u_{\delta(\pi_i(\lambda))}$ can be precisely described as a map on the groupoid $\delta(\Lambda_0^{(I)})\ltimes Y$, see the paragraph before Theorem \ref{thm: coe by isomorphic topological groupoids}. Then 
\begin{align*}
&(\beta^*_{\gamma}(u_{\delta(\pi_i(\lambda))}))(g, y)\\
&=u_{\delta(\pi_i(\lambda))}(\beta_{\gamma^{-1}}(g, y))~\quad\text{(since $\beta^*$ is induced by $\beta$)}\\
&\overset{\eqref{def: definition of beta-action}}{=}u_{\delta(\pi_i(\lambda))}(\delta(\gamma^{-1})g\delta(\gamma), \delta(\gamma^{-1})*y)\\
&=\begin{cases}
1,~&~\text{if $\delta(\pi_i(\lambda))=\delta(\gamma^{-1})g\delta(\gamma)$,}\\
0,~&~\text{otherwise}
\end{cases}\\
&=\begin{cases}
1,~&~\text{if $\delta(\pi_i(\lambda))=g$,}\\
0,~&~\text{otherwise}
\end{cases}~\quad\text{(since $\gamma$ commutes with $\pi_i(\lambda)$)}\\
&=u_{\delta(\pi_i(\lambda))}(g, y).
\end{align*}
Therefore, we have proved that $\beta^*_{\gamma}(u_{\delta(\pi_i(\lambda))})=u_{\delta(\pi_i(\lambda))}$. Hence,
\begin{align*}
S^*_{\gamma}\Phi^*(u_{\delta(\pi_i(\lambda))})
\overset{\eqref{relation: Phi intertwins S and beta}}{=}\Phi^*\beta^*_{\gamma}(u_{\delta(\pi_i(\lambda))})=\Phi^*(u_{\delta(\pi_i(\lambda))}).
\end{align*}
In other words, $\Phi^*(u_{\delta(\pi_i(\lambda))})$ is invariant under the action $S^*|_{Stab(i)}$. Since $\lambda\in \Lambda_0$ is arbitrary, we deduce that $\Phi^*(C_r^*(\delta(\pi_i(\Lambda_0))))$ is contained in the fixed point $C^*$-subalgebra for the action $S^*:~Stab(i)\curvearrowright C^*_r(\Lambda^{(I)}\ltimes X)$. 

From now on, to ease notations, we always identify $S^*$ with the generalized Bernoulli shift action $Stab(i)\curvearrowright \otimes_IC^*_r(\Lambda\ltimes X_0)$ via the isomorphism on the bottom line in the diagram \eqref{figure1}.

Since $\Gamma\curvearrowright I$ satisfies (S2) in the standing assumption, we may apply Lemma \ref{lem: fixed point algebra under the wm assumption} to conclude that
$\Phi^*(C_r^*(\delta(\pi_i(\Lambda_0))))\subset \Pi_i(C^*_r(\Lambda\ltimes X_0))$. See Remark \ref{remark: how to apply the fixed point subalgebra} for explanation.

Finally we can show $\delta$ is injective. 

Let $a\in Ker\delta<\Lambda_0^{(I)}$ (by step 5). Take a finite subset $J=\{i_1,\ldots, i_m\}\subset I$ such that $a\in \Lambda_0^J$. Write $b_k=\delta(\pi_{i_k}(a_{i_k}))$. Then $e=\delta(a)=b_1\cdots b_m$. Thus,
\[1=\Phi^*(u_e)=\Phi^*(u_{b_1})\cdots\Phi^*(u_{b_m}).\]
Note that each $\Phi^*(u_{b_k})$ is a unitary in $\Pi_{i_k}(C_r^*(\Lambda\ltimes X_0))$ as shown just now and the indices $i_k$ are distinct. If such a product of unitaries is equal to 1, we must have $\Phi^*(u_{b_k})=1$ for all $k$. Thus, $u_{b_k}=e$ and hence $\delta(\pi_{i_k}(a_{i_k}))=e$. Taking $g_k\in\Gamma$ such that $g_ki_0=i_k$, then 
\begin{align*}
e=\delta(\pi_{i_k}(a_{i_k}))=\delta(\pi_{g_ki_0}(a_{i_k}))
=\delta(g_k)\delta(\pi_{i_0}(a_{i_k})\delta(g_k^{-1})
=\delta(g_k)a_{i_k}\delta(g_k)^{-1}.
\end{align*}
Here, the last equality holds by the definition of $\delta: G_2\to G$ in step 3.
Therefore, $a_{i_k}=e$ for all $k$, so $a=e$. 

To sum up, we may draw the following commutative diagram to help remembering the situation. Recall that $\omega=\delta\circ \omega_2$ by the construction of $\omega_2$ in step 3. Moreover, $\delta$ is a group isomorphism by step 6 and step 7.
\begin{align}\label{figure2}
\xymatrix{
{\Lambda\wr_I\Gamma\times X}\ar[rr]^{\omega}\ar[d]_{\omega_2}&&G\\
G_2=\Lambda_0\wr_I\Gamma\ar[urr]_{\cong}^{\delta}&&
}
\end{align}

\textbf{Step 8}. We argue that there is an action $\rho: \Lambda_0\curvearrowright X_0$ such that 
the initial action $\Lambda\curvearrowright X_0$ is continuously orbit equivalent to $\rho$ via the identity homeomorphism on $X_0$ such that
$\omega_{i_0}: \Lambda\times X_0\to \Lambda_0=\Lambda_{i_0}$ is the associated orbit cocycle.

For any $\lambda_0\in \Lambda_0$ and $x_0\in X_0$, we can just define 
\begin{align}\label{def: definition of the new action rho}
\rho_{\lambda_0}(x_0)=\lambda x_0,~\text{where $\lambda\in \Lambda$ such that $\omega_{i_0}(\lambda, x_0)=\lambda_0$.}
\end{align}
 Let us show that such a $\lambda$ does exist and is uniquely determined by $x_0$ and $\lambda_0$.

 First, notice that $\Lambda\wr_I\Gamma \ni ga\mapsto \omega_2(ga, x)\in \Lambda_{i_0}\wr_I\Gamma=G_2$ is a bijection for any fixed $x\in X$. The reason is that both $\delta$  and $g\mapsto \omega(g, x)$ are bijections, see step 6, step 7 and \eqref{fact: the new omega(-, x) is a bijection}.
By restricting $\omega_2(-, x)$ to $\Lambda_{(i_0)}$, the image of $\Lambda$ under the inclusion $\Lambda\hookrightarrow\oplus_I\Lambda$ into the $i_0$-th component, we get the following bijection 
\[\Lambda_{(i_0)}\ni \lambda_{(i_0)}\mapsto \omega_2(\lambda_{(i_0)}, x)\overset{\eqref{def: def of omega_2}}{=}\pi_{i_0}(\omega_{i_0}(\lambda,x_{i_0}))\in \pi_{i_0}(\Lambda_{i_0}).\] 
 This is the same as saying $\Lambda\ni \lambda\mapsto \omega_{i_0}(\lambda, x_{i_0})\in \Lambda_{i_0}=\Lambda_0$ is a bijection for any $x_{i_0}\in X_0$. Therefore, there exists a unique $\lambda\in \Lambda$ such that $\omega_{i_0}(\lambda, x_0)=\lambda_0$.

It is routine to check that $\rho$ is continuous and is a free action. Moreover, there is a natural inverse cocycle $\omega_{i_0}': \Lambda_0\times X_0\to \Lambda$ by setting $\omega_{i_0}'(\lambda_0, x_0)=\lambda$, whenever $\omega_{i_0}(\lambda, x_0)=\lambda_0$. Clearly, $(\omega_{i_0}, \omega_{i_0}', id)$ forms a coe coupling between $\Lambda\curvearrowright X_0$ and $\rho:~\Lambda_0\curvearrowright X_0$.

Denote by $G_2=\Lambda_0\wr_I\Gamma\overset{\hat{\star}}{\curvearrowright} X_0^I$ the natural action induced by $\rho: \Lambda_0\curvearrowright X_0$. Recall that $\hat{\star}$ is defined as follows: for each $a=(\pi_i(a_i))_{i\in I}\in \Lambda_0^{(I)}$, $g\in\Gamma$ and $x\in X_0^I$, where $a_i\in \Lambda_0$ for all $i\in I$, we have 
\begin{align}\label{def: definition of the hat-star action}
(g~ \hat{\star} ~x)_i=x_{g^{-1}i},~
(a ~\hat{\star}~ x)_i=\rho_{a_i}(x_i),~\forall~i\in I.
\end{align}

\textbf{Step 9}. We are left to argue that the above action $\hat{\star}$ is topologically conjugate to the initial action $G\overset{*}{\curvearrowright} Y=X_0^I$.

The idea is to check that the two actions $\hat{\star}$ and $\star$ are topologically conjugate via the group isomorphism $\delta: G_2\to G$ and the homeomorphism $\phi: X\rightarrow Y$. Recall that $(\omega, \omega', \phi)$ is a coe coupling for the continuous orbit equivalence between the two initial actions $\Lambda\wr_I\Gamma\overset{\star}{\curvearrowright} X$ and $G\overset{*}{\curvearrowright}Y$. 

Fix any $g\in \Gamma$ and $x\in X$. Clearly, $g~\hat{\star}~ x=g\star x$. 
Recall that $\omega=\delta\circ \omega_2$ (as shown on the diagram \eqref{figure2}) and $\omega_2(g, x)=g$ for all $g\in \Gamma$ by \eqref{def: def of omega_2}. A calculation shows that
\begin{align}\label{eq: last eq1}
\phi(g~\hat{\star}~ x)=\phi(g\star x)=\omega(g, x)*\phi(x)=\delta(\omega_2(g, x))*\phi(x)=\delta(g)*\phi(x).
\end{align}
Next, fix any $a=(a_i)_{i\in I}\in \Lambda_0^{(I)}$. Recall that $\Lambda\wr_I\Gamma\ni a'\mapsto \omega_2(a', x)\to G_2$ is a bijection as observed in step 8. We may take $a'\in \Lambda\wr_I\Gamma$ such that $\omega_2(a', x)=a$. By \eqref{property: property of omega_2}, we actually have $a'\in \Lambda^{(I)}$ and $a_i=\omega_{i_0}(a'_i, x_i)$ for all $i\in I$, where we have written $a'=(a_i')_{i\in I}\in \Lambda^{(I)}$.
Then we can deduce that
\begin{align}\label{eq: last eq2}
\phi(a~\hat{\star}~ x)=\phi(a'\star x)=\delta(a)*\phi(x).
\end{align}
To see this, for every $i\in I$, we compute
\begin{align*}
(a~\hat{\star}~ x)_i&\overset{\eqref{def: definition of the hat-star action}}{=}\rho_{a_i}(x_i)
\\
&=a_i'x_i~\quad\text{(by \eqref{def: definition of the new action rho} and $\omega_{i_0}(a_i', x_i)=a_i$)}\\
&=(a'\star x)_i.
\end{align*}
Hence, $a~\hat{\star}~ x=a'\star x$ and $\phi(a~\hat{\star}~x)=\phi(a'\star x)$. On the other hand,
\begin{align*}
\delta(a)*\phi(x)&=\delta(\omega_2(a', x))*\phi(x)~\quad\text{(since $a=\omega_2(a', x)$)}\\
&=\omega(a', x)*\phi(x)~\quad\text{(since $\omega=\delta\circ\omega_2$)}\\
&=\phi(a'\star x)~\quad\text{(since $(\omega, \omega', \phi)$ is a coe coupling)}.
\end{align*}
Since $G_2$ is generated by $\Gamma$ and $\Lambda_0^{(I)}$, the proof is done by \eqref{eq: last eq1} and \eqref{eq: last eq2}.
\end{proof}

Now, we can prove Theorem \ref{thm: main thm for the whole paper} easily.

\begin{proof}[Proof of Theorem \ref{thm: main thm for the whole paper}]
The left-right wreath product action $\frac{\mathbb{Z}}{p\mathbb{Z}}\wr_{\Gamma_1}(\Gamma_1\times \Gamma_1)\curvearrowright X$ is topologically free and minimal by applying 
Lemma \ref{lem: topological freeness for generalized wreath product actions} and Remark \ref{remark: minimality for generalized wreath product actions}.

By Proposition \ref{prop: left-right translation actions satisfy the standing assumption}, the left-right shift $\Gamma:=\Gamma_1\times \Gamma_1\curvearrowright \Gamma_1:=I$ satisfies the standing assumption. By  Corollary \ref{main corollary} (ii), the generalized full shift $\Gamma_1\times \Gamma_1\curvearrowright X=(\frac{\mathbb{Z}}{p\mathbb{Z}})^{\Gamma_1}$ is a continuous cocycle superrigid action. Hence the three conditions in Theorem \ref{thm: main thm for generalized wreath product actions} are satisfied, and we may apply Corollary \ref{main corollary2} to finish the proof.
\end{proof}

We remark that the action in Theorem \ref{thm: main thm for the whole paper} is not free since the subaction $\Gamma_1\times \Gamma_1\curvearrowright X$ is not free.

\section{Questions}\label{section: ending questions}

Related to this paper, the following questions might be worth studying.

1. Determine conditions on $\Gamma\curvearrowright X_0^I$ such that the conclusion in Theorem \ref{thm: ccsr for generalized full shifts} holds for it.

2. Does the conclusion in Theorem \ref{thm: ccsr for generalized full shifts} still hold true for 
infinite compact bottom space $X_0$?

\subsection*{Acknowledgements}

This work is supported by NSFC grant no. 12001081. We are very grateful to Dr. Daniel Drimbe for explaining the proof of Theorem 5.1 in \cite{dv} patiently to us and answering our questions. We also thank Professor Nhan-Phu Chung for the collaboration on \cite{cj, cj_etds}, which contain some initial ideas for proving Theorem \ref{thm: ccsr for generalized full shifts}, see Remark \ref{remark: relation to our previous work}. We also thank the referee a lot for reading the paper very carefully, pointing out lots of typos and inaccuracies and providing many helpful suggestions for improving the exposition greatly.


\begin{bibdiv}
\begin{biblist}

\bib{AP}{book}{
    AUTHOR = {Anantharaman-Delaroche, C.},
    author={Popa, S.},
     TITLE = {An introduction to II$_1$ factors},
 status = {book preprint, available at \url{https://www.math.ucla.edu/~popa/Books/IIunV15.pdf}},
      YEAR = {2018},
     }

\bib{BTD}{article}{
   author={Bowen, L.},
   author={Tucker-Drob, R.},
   title={Superrigidity, measure equivalence, and weak Pinsker entropy},
   journal={Groups Geom. Dyn.},
   volume={16},
   date={2022},
   number={1},
   pages={247--286},}

\bib{BT}{article}{
   author={Boyle, M.},
   author={Tomiyama, J.},
   title={Bounded topological orbit equivalence and $C^*$-algebras},
   journal={J. Math. Soc. Japan},
   volume={50},
   date={1998},
   number={2},
   pages={317--329},}

\bib{CK15}{article}{
   author={Chifan, I.},
   author={Kida, Y.},
   title={$OE$ and $W^*$ superrigidity results for actions by surface braid
   groups},
   journal={Proc. Lond. Math. Soc. (3)},
   volume={111},
   date={2015},
   number={6},
   pages={1431--1470},}

\bib{cj}{article}{
author={Chung, N.-P.},
author={Jiang, Y.},
title={Continuous Cocycle Superrigidity for Shifts and Groups with One End},
journal={Math. Ann.},
   volume={368},
   date={2017},
   number={3-4},
   pages={1109--1132},
}

\bib{cj_etds}{article}{
   author={Chung, N.-P.},
   author={Jiang, Y.},
   title={Divergence, undistortion and H\"{o}lder continuous cocycle
   superrigidity for full shifts},
   journal={Ergodic Theory Dynam. Systems},
   volume={41},
   date={2021},
   number={8},
   pages={2274--2293},}

\bib{cohen}{article}{
   author={Cohen, D. B.},
   title={Continuous cocycle superrigidity for the full shift over a
   finitely generated torsion group},
   journal={Int. Math. Res. Not. IMRN},
   date={2020},
   number={6},
   pages={1610--1620},}

\bib{CM}{article}{
   author={Cortez, M. I.},
   author={Medynets, K.},
   title={Orbit equivalence rigidity of equicontinuous systems},
   journal={J. Lond. Math. Soc. (2)},
   volume={94},
   date={2016},
   number={2},
   pages={545--556},}

\bib{Dr15}{article}{
   author={Drimbe, D.},
   title={Cocycle and orbit equivalence superrigidity for coinduced actions},
   journal={Ergodic Theory Dynam. Systems},
   volume={38},
   date={2018},
   number={7},
   pages={2644--2665},}
   
   \bib{Dr20}{article}{
   author={Drimbe, D.},
   title={Orbit equivalence rigidity for product actions},
   journal={Comm. Math. Phys.},
   volume={379},
   date={2020},
   number={1},
   pages={41--59},}

\bib{DIP}{article}{
author={Drimbe, D.},
author={Ioana, A.},
author={Peterson, J.},
title={Cocycle superrigidity for profinite actions of irreducible lattices},
status={Groups Geom. Dyn., to appear},
}


\bib{dv}{article}{
author={Drimbe, D.},
author={Vaes, S.},
title={Superrigidity for dense subgroups of Lie groups and their actions on homogeneous spaces},
status={Math. Ann., to appear},
}

\bib{dye}{article}{
   author={Dye, H. A.},
   title={On groups of measure preserving transformations. I},
   journal={Amer. J. Math.},
   volume={81},
   date={1959},
   pages={119--159},}

\bib{Fur99a}{article}{
   author={Furman, A.},
   title={Gromov's measure equivalence and rigidity of higher rank lattices},
   journal={Ann. of Math. (2)},
   volume={150},
   date={1999},
   number={3},
   pages={1059--1081},}

\bib{Fur99b}{article}{
   author={Furman, A.},
   title={Orbit equivalence rigidity},
   journal={Ann. of Math. (2)},
   volume={150},
   date={1999},
   number={3},
   pages={1083--1108},}


\bib{GITD16}{article}{
   author={Gaboriau, D.},
   author={Ioana, A.},
   author={Tucker-Drob, R.},
   title={Cocycle superrigidity for translation actions of product groups},
   journal={Amer. J. Math.},
   volume={141},
   date={2019},
   number={5},
   pages={1347--1374},}


\bib{GPS}{article}{
   author={Giordano, T.},
   author={Putnam, I. F.},
   author={Skau, C. F.},
   title={$\Bbb{Z}^d$-odometers and cohomology},
   journal={Groups Geom. Dyn.},
   volume={13},
   date={2019},
   number={3},
   pages={909--938},}

\bib{GH}{article}{
author={Guirardel, V.},
author={Horbez, C.},
title={Measure equivalence rigidity of $Out(F_n)$},
status={arXiv: 2103.03696},
}

\bib{HenH}{article}{
author={Hensel, S.},
author={Horbez, C.},
title={Measure equivalence rigidity of the handlebody groups},
status={arXiv: 2111.10064},
}

\bib{HH1}{article}{
   author={Horbez, C.},
   author={Huang, J.},
   title={Measure equivalence classification of transvection-free
   right-angled Artin groups},
   language={English, with English and French summaries},
   journal={J. \'{E}c. polytech. Math.},
   volume={9},
   date={2022},
   pages={1021--1067},}

\bib{HH2}{article}{
   author={Horbez, C.},
   author={Huang, J.},
title={Orbit equivalence rigidity of irreducible actions of right-angled Artin groups},
status={arXiv: 2110.04141},
}
\bib{HH3}{article}{
author={Horbez, C.},
author={Huang, J.},
title={Measure equivalence rigidity among the Higman groups},
status={arXiv: 2206.00884},   
}


\bib{Io08}{article}{
   author={Ioana, A.},
   title={Cocycle superrigidity for profinite actions of property (T)
   groups},
   journal={Duke Math. J.},
   volume={157},
   date={2011},
   number={2},
   pages={337--367},}   

\bib{Io14}{article}{
   author={Ioana, A.},
   title={Strong ergodicity, property (T), and orbit equivalence rigidity
   for translation actions},
   journal={J. Reine Angew. Math.},
   volume={733},
   date={2017},
   pages={203--250},}

   

\bib{jiang_pams}{article}{
   author={Jiang, Y.},
   title={Continuous cocycle superrigidity for coinduced actions and
   relative ends},
   journal={Proc. Amer. Math. Soc.},
   volume={147},
   date={2019},
   number={1},
   pages={315--326},}

\bib{jiang_ggd}{article}{
author={Jiang, Y.},
title={On continuous orbit equivalence rigidity for virtually cyclic group actions},
status={Groups Geom. Dyn., to appear},
}


\bib{KR_II}{book}{
   author={Kadison, R. V.},
   author={Ringrose, J. R.},
   title={Fundamentals of the theory of operator algebras. Vol. II},
   series={Pure and Applied Mathematics},
   volume={100},
   note={Advanced theory},
   publisher={Academic Press, Inc., Orlando, FL},
   date={1986},
   pages={i--xiv and 399--1074},}

\bib{KL_book}{book}{
   author={Kerr, D.},
   author={Li, H.},
   title={Ergodic theory Independence and dichotomies},
   series={Springer Monographs in Mathematics},
   publisher={Springer, Cham},
   date={2016},
   pages={xxxiv+431},
   }
   
\bib{Ki06}{article}{
   author={Kida, Y.},
   title={Measure equivalence rigidity of the mapping class group},
   journal={Ann. of Math. (2)},
   volume={171},
   date={2010},
   number={3},
   pages={1851--1901},}   
   
   \bib{Ki09}{article}{
   author={Kida, Y.},
   title={Rigidity of amalgamated free products in measure equivalence},
   journal={J. Topol.},
   volume={4},
   date={2011},
   number={3},
   pages={687--735},}
   
\bib{LX}{article}{
   author={Li, X.},
   title={Continuous orbit equivalence rigidity},
   journal={Ergodic Theory Dynam. Systems},
   volume={38},
   date={2018},
   number={4},
   pages={1543--1563},}
   
\bib{LX_agt}{article}{
   author={Li, X.},
   title={Dynamic characterizations of quasi-isometry and applications to
   cohomology},
   journal={Algebr. Geom. Topol.},
   volume={18},
   date={2018},
   number={6},
   pages={3477--3535},}   
   
   
   
   
   
\bib{liv1}{article}{
   author={Liv\v{s}ic, A. N.},
   title={Certain properties of the homology of $Y$-systems},
   language={Russian},
   journal={Mat. Zametki},
   volume={10},
   date={1971},
   pages={555--564},}   
   
   \bib{liv2}{article}{
   author={Liv\v{s}ic, A. N.},
   title={Cohomology of dynamical systems},
   language={Russian},
   journal={Izv. Akad. Nauk SSSR Ser. Mat.},
   volume={36},
   date={1972},
   pages={1296--1320},}
   
\bib{MS02}{article}{
   author={Monod, N.},
   author={Shalom, Y.},
   title={Orbit equivalence rigidity and bounded cohomology},
   journal={Ann. of Math. (2)},
   volume={164},
   date={2006},
   number={3},
   pages={825--878},}   
   
\bib{np}{article}{
   author={Nicol, M.},
   author={Pollicott, M.},
   title={Measurable cocycle rigidity for some non-compact groups},
   journal={Bull. London Math. Soc.},
   volume={31},
   date={1999},
   number={5},
   pages={592--600},}   
   
  \bib{pp}{article}{
   author={Parry, W.},
   author={Pollicott, M.},
   title={The Liv\v{s}ic cocycle equation for compact Lie group extensions of
   hyperbolic systems},
   journal={J. London Math. Soc. (2)},
   volume={56},
   date={1997},
   number={2},
   pages={405--416},} 
   
\bib{PS09}{article}{
   author={Peterson, J.},
   author={Sinclair, T.},
   title={On cocycle superrigidity for Gaussian actions},
   journal={Ergodic Theory Dynam. Systems},
   volume={32},
   date={2012},
   number={1},
   pages={249--272},}   
   
\bib{pw}{article}{
   author={Pollicott, M.},
   author={Walkden, C. P.},
   title={Liv\v{s}ic theorems for connected Lie groups},
   journal={Trans. Amer. Math. Soc.},
   volume={353},
   date={2001},
   number={7},
   pages={2879--2895},}  
   
\bib{py}{article}{
   author={Pollicott, M.},
   author={Yuri, M.},
   title={Regularity of solutions to the measurable Livsic equation},
   journal={Trans. Amer. Math. Soc.},
   volume={351},
   date={1999},
   number={2},
   pages={559--568},}   
   
  \bib{popa_cocycle1}{article}{
   author={Popa, S.},
   title={Cocycle and orbit equivalence superrigidity for malleable actions
   of $w$-rigid groups},
   journal={Invent. Math.},
   volume={170},
   date={2007},
   number={2},
   pages={243--295},}
   
   \bib{popa_cocycle2}{article}{
   author={Popa, S.},
   title={On the superrigidity of malleable actions with spectral gap},
   journal={J. Amer. Math. Soc.},
   volume={21},
   date={2008},
   number={4},
   pages={981--1000},}

\bib{PV_adv}{article}{
   author={Popa, S.},
   author={Vaes, S.},
   title={Strong rigidity of generalized Bernoulli actions and computations
   of their symmetry groups},
   journal={Adv. Math.},
   volume={217},
   date={2008},
   number={2},
   pages={833--872},}
   
   \bib{PV08}{article}{
   author={Popa, S.},
   author={Vaes, S.},
   title={Cocycle and orbit superrigidity for lattices in ${\rm SL}(n,\Bbb
   R)$ acting on homogeneous spaces},
   conference={
      title={Geometry, rigidity, and group actions},
   },
   book={
      series={Chicago Lectures in Math.},
      publisher={Univ. Chicago Press, Chicago, IL},
   },
   date={2011},
   pages={419--451},}

\bib{quas}{article}{
   author={Quas, A. N.},
   title={Rigidity of continuous coboundaries},
   journal={Bull. London Math. Soc.},
   volume={29},
   date={1997},
   number={5},
   pages={595--600},}


\bib{Ren}{book}{
   author={Renault, J.},
   title={A groupoid approach to $C^{\ast} $-algebras},
   series={Lecture Notes in Mathematics},
   volume={793},
   publisher={Springer, Berlin},
   date={1980},
   pages={ii+160},}

\bib{Rudin}{book}{
   author={Rudin, W.},
   title={Real and complex analysis},
   edition={3},
   publisher={McGraw-Hill Book Co., New York},
   date={1987},
   pages={xiv+416},}

\bib{SC95}{article}{
   author={Schmidt, K.},
   title={The cohomology of higher-dimensional shifts of finite type},
   journal={Pacific J. Math.},
   volume={170},
   date={1995},
   number={1},
   pages={237--269},}

\bib{SC}{article}{
   author={Schmidt, K.},
   title={Remarks on Liv\v{s}ic' theory for nonabelian cocycles},
   journal={Ergodic Theory Dynam. Systems},
   volume={19},
   date={1999},
   number={3},
   pages={703--721},}
   
  \bib{SSW}{collection}{
   author={Sims, A.},
   author={Szab\'{o}, G.},
   author={Williams, D.},
   title={Operator algebras and dynamics: groupoids, crossed products, and
   Rokhlin dimension},
   series={Advanced Courses in Mathematics. CRM Barcelona},
   publisher={Birkh\"{a}user/Springer, Cham},
   date={2020},
   pages={x+163},} 
   
\bib{Tak_book}{book}{
   author={Takesaki, M.},
   title={Theory of operator algebras. III},
   series={Encyclopaedia of Mathematical Sciences},
   volume={127},
   note={Operator Algebras and Non-commutative Geometry, 8},
   publisher={Springer-Verlag, Berlin},
   date={2003},
   pages={xxii+548},}   
   
   \bib{TD14}{article}{
   author={Tucker-Drob, R.},
   title={Invariant means and the structure of inner amenable groups},
   journal={Duke Math. J.},
   volume={169},
   date={2020},
   number={13},
   pages={2571--2628},}

 \bib{wp}{article}{
   author={Walters, P.},
   title={Topological Wiener-Wintner ergodic theorems and a random $L^2$
   ergodic theorem},
   journal={Ergodic Theory Dynam. Systems},
   volume={16},
   date={1996},
   number={1},
   pages={179--206},}

\end{biblist}
\end{bibdiv}
\end{document}